\documentclass{amsart}

\usepackage{latexsym}
\usepackage{enumerate}
\usepackage{amssymb}
\usepackage{amsfonts}
\usepackage{amsmath}
\usepackage{mathrsfs}
\usepackage{amsthm}
\usepackage{geometry}
\usepackage[colorlinks,
linkcolor=red,
anchorcolor=blue,
citecolor=green
]{hyperref}
\usepackage{cleveref}
\usepackage[all]{xy}

\newcommand{\mclass}{\mathscr{M}_{holo}^{m,\alpha } (n,Q,r)}

\newcommand{\partialbar}{\bar{\partial}}
\newcommand{\zbar}{\bar{z}}

\newcommand{\Vol}{\mathrm{Vol}}
\newcommand{\sq}{\backslash}

\newtheorem{thm}{Theorem}[section]
\newtheorem{definition}[thm]{Definition}

\newtheorem{prop}[thm]{Proposition}
\newtheorem{coro}[thm]{Corollary}
\newtheorem{lmm}[thm]{Lemma}
\newtheorem{exap}[thm]{Example}

\theoremstyle{remark} 
\newtheorem*{rmk}{Remark}

\title{On the convergence rate of Bergman metrics}

\author{Shengxuan Zhou}
\address{Beijing International Center for Mathematical Research\\
Peking University\\
Beijing\\ 100871\\ China}
\email{zhoushx19@pku.edu.cn}

\begin{document}

\begin{abstract}
We study the convergence rate of Bergman metrics on the class of polarized pointed K\"ahler $n$-manifolds $(M,L,g,x)$ with $\Vol\left( B_1 (x) \right) >v $ and $|\sec |\leq K $ on $M$. Relying on Tian's peak section method \cite{tg1}, we show that the $C^{1,\alpha }$ convergence of Bergman metrics is uniform. In the end, we discuss the sharpness of our estimates.
\end{abstract}

\maketitle

\tableofcontents

\section{Introduction}

Let $(M,g )$ be an $n$-dimensional complete K\"ahler manifold, $L$ be a positive on $M$ equipped with a hermitian metric $h$ whose curvature form is $2\pi\omega $. The $L^2$ orthonormal basis of $H^0\left( M,L^m \right)$ will induce canonical embeddings $\varphi_m $ of $M$ into $\mathbb{C}P^{N_m -1} $, where $N_m = \mathrm{dim} H^0\left( M,L^m \right) $. The pullbacks of the $\frac{1}{m} -$multiple of Fubini-Study metrics $g_m = \frac{1}{m} \varphi_m^* g_{FS}$ are usually called Bergman metrics.

A natural question is to compare the Bergman metrics with the original K\"ahler metrics. In the pioneering work \cite{tg1}, Tian used his peak section method to prove that Bergman metrics converge to the original polarized metric in the $C^2$-topology. By the similar method, Ruan \cite{wdr1} proved that this convergence is $C^\infty $. Later, Zelditch \cite{sz1}, also Catlin \cite{cat1} independently, used the Szegö kernel to obtain an alternative proof of the $C^\infty $-convergence of Bergman metrics and they gave the asymptotic expansion of Bergman kernel, which is the potential of Bergman metric. This expansion can be also obtained by Tian's peak section method (see \cite{liuzql1}) and is often called Tian-Yau-Zelditch expansion. By using the heat kernel, Dai-Liu-Ma \cite{dailiuma1} gave another proof of the Tian–Yau–Zelditch expansion, and moreover, they also considered the asymptotic behavior of Bergman kernels on symplectic manifolds and Kähler orbifolds (see also Ma-Marinescu’s book \cite{mamari1}). There are many important applications of using Bergman metrics to approximate a given Kähler metric, for example, see \cite{don1}. 
      
In this paper, we study the problem on dependence of the convergence rate. We focus on Bergman metrics in this paper. For Bergman kernel, please view the discussion in \cite{sxz1}. Our first result is an estimate of the $C^1$-convergence rate of Bergman metrics stated as follows.

\begin{thm}
\label{thm1}
Let $(M,g)$ be a polarized K\"ahler manifold. Assume that there are constants $K,v >0$ such that $\left| \sec \right|\leq K$ on $M$ and $\Vol \left( B_{1}(x_0 ) \right) >v ,$ for $ x_0 \in M$. Then we have constants $m_0 =m_0 (K,v)\in\mathbb{N}$ and $C=C(K,v ) >0$ such that
$$ \sqrt{m} \left\Vert \nabla g_{m} (x_0 )  \right\Vert +m  \left\Vert g_{m} (x_0 ) -g(x_0 ) \right\Vert   \leq C ,\; \forall m>m_0 ,$$
where $||\cdot ||$ is the norm of tensors which induced by $g$.
\end{thm}

\begin{rmk}
We emphasize that the constants in the estimate are uniform for all $(M,g)$ satisfying the assumptions there. For a fixed $(M,g)$, Ruan (\cite{wdr1}) proved that $ \left\Vert g_{m} -g \right\Vert_{C^\infty } = O\left( \frac{1}{m} \right)  $. But in the uniform sense, even if we consider the $C^1 $ case, we cannot prove that the difference between metrics is $O\left( \frac{1}{m} \right)$. In fact, the $O\left(\frac{1}{\sqrt{m}} \right)$ rate in Theorem \ref{thm1} is sharp. See example \ref{sharp}.
\end{rmk}

Our second main result is about $C^{1,\alpha }$-convergence of Bergman metrics:

\begin{thm}
\label{thm2}
Let $(M,g)$ be a polarized K\"ahler manifold. Assume that there are constants $K,v >0$ such that $\left| \sec \right|\leq K$ on $M$, and $\Vol \left( B_{1}(x_0 ) \right) >v ,$ for $ x_0 \in M$. Then there are constants $r$ and $C$ depend only on $K$ and $v$, such that there exists a holomorphic chart $(z_1 ,\cdots ,z_n )$ containing $B_r (x_0 )$ with
$$ e^{-C} g_{\mathbb{C}^n } \leq g \leq e^{C} g_{\mathbb{C}^n } ,$$
and for each $\alpha \in (0,1)$, we have a constant $m_0 = m_0 (K,v ,\alpha )$, such that
$$  \left\Vert  g_{i\bar{j}, m} - g_{i\bar{j}}  \right\Vert_{C^{1,\alpha}}  \leq Cm^{\frac{-1+\alpha }{2}}  \left| \log(m) \right|^{\alpha} ,\; \forall m>m_0 ,$$
for $i,j=1,\cdots ,n$. Here $g_{\mathbb{C}^n } $ is the normal flat metric on $\mathbb{C}^n $, $\left\Vert  \cdot  \right\Vert_{C^{1,\alpha}}$ is the $C^{1,\alpha}$-norm on the chart, $g_{i\bar{j} ,m} = g_m \left( \frac{\partial }{\partial z_i} ,\frac{\partial }{\partial \zbar_j }  \right) $, and $g_{i\bar{j} } = g \left( \frac{\partial }{\partial z_i} ,\frac{\partial }{\partial \zbar_j }  \right) $.
\end{thm}

As a byproduct, we also obtain an estimate for the $W^{2,p}$ convergence. In addition, uniform $W^{2,1}$ convergence is not generally true (see Section \ref{exp}).

For the proof of Theorem \ref{thm1} and \ref{thm2}, we will use the holomorphic version of the Cheeger-Gromov convergence theory and Tian's peak section method. By Taylor expansion, we use the peak sections to approximate the holomorphic section with the largest norm at a given point. Since the second-order expansion is not uniform, we also need to estimate some special Fourier coefficients on distinguish boundaries of polydiscs.

This paper is organized as follows. In Section \ref{preliminaries}, we collect some preliminary results. We will give the method of constructing local coordinates in Section \ref{cuhc}. Then we will make some estimates about Tian's peak sections in Section \ref{tps}. We prove Theorem \ref{thm1} and \ref{thm2} in Section \ref{pe} and Section \ref{lhestm}. In Section \ref{exp}, we will give some examples about surfaces to illustrate what happens if some conditions are removed or if we consider faster convergence. For convenience, we save some details of it in Appendix \ref{HNoKM}. We also put some ODE estimates in Appendix \ref{pode}.

\vspace{0.2cm}

\textbf{Acknowledgement.} The author wants to express his deep gratitude to Professor Gang Tian for suggesting this problem and constant encouragement. He also thanks Zexing Li for reading the first version carefully, pointing out mistakes and typos, and giving many suggestions about writing the article.

\section{Preliminaries}
\label{preliminaries}
First, recall some basic notations in K\"ahler geometry. Let $\pi :L\to M$ be a holomorphic line bundle on the $n$-dimensional K\"ahler manifold $M$, and let $h$ be a hermitian metric on $L$. Then the curvature of $(L,h)$ is the $(1,1)$ form $Ric(h)= -\sqrt{-1} \partial\bar{\partial} \log h$. For each ample line bundle $L$ on $M$, we can find a hermitian metric $h$ on $L$ such that $\omega = \frac{1}{2\pi } Ric(h)>0$. Then we say that $(M,\omega )$ is a polarized K\"ahler manifold with polarized K\"ahler metric $\omega $. The $L^2$ orthonormal basis of $H^0\left( M,L^m \right)$ will induce canonical embeddings $\varphi_m $ of $M$ into $\mathbb{C}P^{N_m -1} $, where $N_m = \mathrm{dim} H^0\left( M,L^m \right) $. The pullbacks of the $\frac{1}{m} -$multiple of Fubini-Study metrics $\omega_m = \frac{1}{m} \varphi_m^* \omega_{FS}$ are usually called Bergman metrics. Let $\left\lbrace S^m_j \right\rbrace_{j=0}^{N_m -1} $ be a $L^2 $ orthonormal basis of $H^0 \left( M,L^m \right) $. Then it's obvious that
\begin{displaymath}
\omega_m = \frac{\sqrt{-1}}{2\pi m} \partial\partialbar \log\left( \sum_{j=0}^{N_m -1} \left| S^m_j (x) \right|^2 \right) .
\end{displaymath} 
We introduce the holomorphic version of Cheeger-Gromov $C^{m,\alpha }$-norm for K\"ahler manifolds now. 
\begin{definition}[Holomorphic Norms]
\label{holonorm}
Let $(M,g,x)$ be a pointed K\"ahler manifold. We say that the holomorphic $C^{m,\alpha}$-$norm\; on\; the\; scale\; of\; r$ at $x$: 
$$\left\Vert (M,g,x) \right\Vert^{holo}_{C^{m,\alpha} ,r} \leq Q,$$
provided there exists a biholomorphic chart $\phi : \left( B_r (0),0 \right) \subset \mathbb{C}^{n} \to   (U,x)\subset M$ such that

\begin{eqnarray}
& & |D\phi |\leq e^Q \textrm{ on } B_r (0) \textrm{ and } \left| D\phi^{-1} \right| \leq e^Q \textrm{ on } U. \\
& & \textrm{For all multi-indices } I\textrm{ with }0\leq |I|\leq m,\\
& & \quad\quad\quad\quad\quad r^{|I|+\alpha} \left\Vert D^{I}g_{i\bar{j}} \right\Vert_{\alpha} \leq Q.\nonumber
\end{eqnarray}

Globally we define
$$ \left\Vert (M,g) \right\Vert^{holo}_{C^{m,\alpha} ,r}=\sup_{x\in M} \left\Vert (M,g,x) \right\Vert^{holo}_{C^{m,\alpha} ,r}.$$

\end{definition}

Then we state the H\"omander's $L^2$ theory:
\begin{prop}
\label{l2m}
Let $(M,\omega )$ be a connected but not necessarily complete K\"ahler manifold with dim$M=n$. Assume that $M$ is Stein if it isn't compact. Let $(L,h)$ be a hermitian holomorphic line bundle, and let $\psi\in L^{1}_{loc}(M)$ be a weight function on $M$. Suppose that $$\sqrt{-1}\partial\bar{\partial} \psi +  Ric(\omega )+Ric(h) \geq \gamma\omega_{g} $$
for some positive continuous function $\gamma $ on $M$. Then for any $L$-valued $(0,1)$-form $\zeta\in L^{2}$ on $M$ with $\bar{\partial} \zeta =0$ and $\int_{M} ||\zeta ||^{2} e^{-\psi} dV_g $ finite, then there exists an $L$-valued function $u\in L^{2}$ such that $\partialbar u=\zeta$ and $$\int_{M} ||u||^{2} e^{-\psi } dV_{g} \leq \int_{M} \gamma^{-1} ||\zeta ||^{2} e^{-\psi} dV_{g} .$$ 
\end{prop}

The proof can be found in \cite{dm1}. By the theory of elliptic equations, we can choose the solution $u\in C^{k+2,\alpha}$(resp. $W^{k+2,p}$), if $\zeta\in C^{k+1,\alpha}$(resp. $W^{k+1,p}$), for $k\geq 0$.

In this paper, the notation $(M,g,J,x)$ means that a pointed K\"ahler manifold with K\"ahler metric $g$, complex structure $J$, and $x\in M$. For simplicity, we denote it by $(M,g,x)$ if we don't emphasize the complex structure.

\section{Construct Uniform Holomorphic Charts}
\label{cuhc}

Let $M$ be an $n$-dimensional algebraic manifold with a polarization $L$, and let $g$ be a polarized K\"ahler metric with respect to $L$, i.e., $\omega_g\in c_1 (L)\in H^2 (M,\mathbb{Z})$. Then there exists a hermitian metric $h$ on $L$ such that $ \frac{1}{2\pi} Ric(h)=\omega_g$. In order to use the Schauder interior estimates, we introduce some norms here.

\begin{definition}[Interior norms]
Let $U\subset \mathbb{R}^n $ be a domain, for each function $f:U\to \mathbb{C} $, we define
\begin{eqnarray*}
[f]^*_{k,0;U} & = & \sup_{x\in U} dist(x,\partial U)^{k} \left| \nabla^k f(x) \right| ,\\
{[f]}^*_{k,\alpha;U} & = & \sup_{x\neq y\in U} \min\left\lbrace dist(x,\partial U) ,dist(y,\partial U)  \right\rbrace^{k} \frac{ \left| \nabla^k f(x) - \nabla^k f(y) \right| }{|x-y|^\alpha } ,\\
\left\Vert f \right\Vert^{*}_{k,\alpha;U} & = & \sum_{j=0}^{k} [f]^*_{j,0;U} + {[f]}^*_{k,\alpha;U} ,
\end{eqnarray*}
where $k\in\mathbb{Z}_+ $, and $\alpha\in (0,1)$.
\end{definition}

We construct holomorphic charts with uniform size now. 

\begin{prop}
\label{holochart}
Let $(M,g,L,h)$ be given as at above, $x \in M$. Suppose that $r\in \left( 0,1 \right]$, $Q>0$ and $\left\Vert (M,g,x) \right\Vert^{holo}_{C^{1,\alpha} ,r} \leq Q $. Then there exists a holomorphic chart $\phi : B_{\delta r} (0 )\to B_r (x )\subset M$, which satisfies the conditions in Definition~\ref{holonorm} with constant $2Q$, and we can find a holomorphic frame $e_L$ of $L$ on $\phi \left( B_{\delta r}(0) \right)$ satisfying that
\begin{eqnarray*}
g_{i\bar{j}} (0) & = & \delta_{ij} ,\\
dg_{i\bar{j}} (0) & = & 0,\\
a(0) & = & 1, \\
 \frac{\partial^{|I|} a}{\partial z^I }  (0) & = & 0,\\
\left\Vert a \right\Vert^{* }_{3,\alpha ;B_{\delta  r}(0 )} & \leq & Cr^{2 } ,
\end{eqnarray*}
for each milti-index $I$ with $|I|\leq 3$, where $g_{i\bar{j}} =g\left( d\phi\left( \frac{\partial}{\partial z_i}\right),d\phi\left(\frac{\partial}{\partial \bar{z}_i}\right)\right)$,  $a=h(e_L ,e_L)$, $\delta$ and $C$ are positive constants depend only on $n,Q$, and $ \left\Vert \cdot \right\Vert^{*}_{k,\alpha;U } $ is the interior norm on the domain $U$.
\end{prop}

\begin{proof}
Without loss of generality, we can assume that $r=1$. By the definition of $\left\Vert (M,g) \right\Vert^{holo}_{C^{1,\alpha} ,r}$, we can find a holomorphic chart $$\phi_{0} :B_{\epsilon } (0 )\to B_r (x) $$ which satisfies the conditions in Definition~\ref{holonorm} with constant $2Q>0$. Replacing the ball $B_{\epsilon } (0 )$ by a polydisc $$U=D_{n^{-1} \epsilon } \left( 0 \right) \times ...\times D_{n^{-1} \epsilon } \left( 0 \right) =D_{n^{-1} \epsilon } \left( 0 \right)^{n},$$ then $H^{1}(U,\mathcal{O}^* ) = 0$ shows that $L\big|_U$ is a trivial bundle, and hence we can choose a holomorphic frame $e_0$ on a smaller ball $ B_{n^{-2} \epsilon} \left( 0 \right) \subset U$. Without loss of generality, we can assume that $L\big|_{B_\epsilon (0 )} $ is trivial, and $e_0 \in H^{0} \left( B_\epsilon (0 ) ) , L \right) $. If necessary, we will appropriately shrink $\epsilon $.

Combining the H\"ormander $L^2$-estimate and the zig-zag argument(see Proposition 8.5 in \cite{bot1}), we can find a real function $f\in C^{\infty} \left( B_\epsilon (0 )  \right) $ such that $$ \int_{B_\epsilon (0 ) } |f|^2 e^{-|z|^2 } dV_{\mathbb{C}^n } \leq C_0 \epsilon^4 \int_{B_\epsilon (0 ) } e^{-|z|^2 } dV_{\mathbb{C}^n } ,$$ where $C_0=C_0 (n)$ is a constant, and $\frac{-\sqrt{-1} }{2\pi } \partial\partialbar f = \omega_{g} $. By the $C^0 $ estimate of solution of Poisson's equation, we can assume that $$\sup_{ B_\epsilon (0 ) } \left| f \right| \leq C_1 \epsilon^{3+\alpha } ,$$ where $C_1 =C_1 (n) $ is a constant.

It is clear that $\partial\partialbar f = \partial\partialbar \log\left( a_0  \right) $, where $a_0 = h\left( e_0 ,e_0 \right) $. Then $ f - \log\left( a_0  \right) $ is a pluriharmonic function, and hence we have a holomorphic function $\psi $ which satisfies that $ Re(\psi ) =  f - \log\left( a_0  \right) $. Let $e_1 = e^{\frac{\psi}{2}} e_0 $, we have $\log\left( h(e_1 ,e_1) \right) = f $.

Next, Schauder's interior estimate(\cite{gilt1}, Theorem 6.2) implies that 
\begin{eqnarray*}
\left\Vert f \right\Vert^{*}_{3,\alpha ; B_\epsilon (0 ) } & \leq & C_{2} \epsilon^2 \left\Vert \Delta f  \right\Vert^{*}_{1, \alpha ; B_\epsilon (0 )} +C_3 \sup_{B\left( 0,\epsilon \right)}\left| f \right| \; \leq \;  C_{4}  \epsilon^{2 } ,
\end{eqnarray*}
where $C_2 ,C_3 , C_4 >0$ are constants depending only on $n $, $Q$.

By the K\"ahler conditions $d\omega_g =0$, we can make $B_{\delta r} (0)$ satisfy the equation $dg_{i\bar{j}} (0)=0$ through a biholomorphic mapping. Likewise, we can assume that $a$ satisfies the vanishing properties at $0\in B_{\delta r} (0) $.
\end{proof}

\begin{rmk}
If $\left\Vert (M,g,x) \right\Vert^{holo}_{C^{k,\alpha} ,r} \leq Q $ for some $k\in\mathbb{N} $, then we can assume that $\left\Vert a \right\Vert^{* }_{k+2,\alpha ;B_\delta (0 )}  \leq  C\delta^2$ and $\frac{\partial^{|I|} a}{\partial z^I }  (0) =0 $, for each multi-index $I$ with $|I|\leq k+2$.
\end{rmk}

\section{Estimate Tian's peak sections}
\label{tps}
In this section, we will make some estimates about Tian's peak sections.

Let $(M,g)$ be an $n$-dimensional algebraic manifold with K\"ahler metric $g$ and a polarization $(L,h)$ such that $\frac{1}{2\pi} Ric(h) = \omega_g$. Fix a local coordinate $\left( z_1 ,...,z_n \right)$ defined on an open neighborhood $U$ around $x_0 \in M$. Define $|z|=\sqrt{\sum_{j=1}^{n} |z_j|^2}$ for $z\in U$.

We assume that $\left\Vert (M,g,x_0 ) \right\Vert^{holo}_{C^{1,\alpha} ,r} \leq Q $ for some $r,Q>0$, $Ric(g)\geq tg$ for some $t\leq 0$, and the local coordinate is the one we constructed in Proposition \ref{holochart}. Now we can construct the peak sections.

\begin{lmm}[\cite{tg1}, Lemma 1.2]
\label{peaksec}
For an $n$-tuple of integers $P=\left( p_1 ,p_2 ,...,p_n \right)\in \mathbb{Z}_{+}^{n}$ and an integer $p'> p=\sum_{j=1}^{n} p_j $, then we can find constant $m_0$ which depends on $t$, $n$, $p$, $p'$, $Q$, and there exists another constant $C_0 $ which depends on $n$, $p$, $p'$, $Q$, such that for each $m>\max \left\lbrace m_0 , \frac{|\log(r)|^2 }{r^2} \right\rbrace $, there are sequences $a_m$ and $b_m $, smooth $L$-valued sections $\varphi_m $, and holomorphic global sections $S_m $ in $H^{0} \left( M,L^m \right)$ satisfying
\begin{eqnarray}
\int_M \left\Vert \varphi_{m} \right\Vert_{h^m}^{2} dV_g &\leq & \frac{C_0}{m^{8p' +2n}} ,\\
\int_{M} \left\Vert S_m \right\Vert_{h^m}^{2} dV_g & = & 1, \\
\int_{M\big\sq\left\lbrace |z|\leq \frac{\log(m) }{\sqrt{m}} \right\rbrace }\left\Vert S_m \right\Vert_{h^m}^{2} dV_g  & \leq &   \frac{C_0}{m^{2p'}}  ,
\end{eqnarray}
and locally at $x_0$,
\begin{align}
S_m (z)  =  \lambda_{\left( p_1 ,p_2 ,...,p_n \right)} \left( 1+ \frac{a_m}{m^{2p'}} \right) \left( z_1^{p_1} \cdots z_n^{p_n} +\varphi_m \right) e_L^m ,
\end{align}
where $||\cdot ||_{h^m}$ is the norm on $L^m$ given by $h^m$, $|a_m |\leq C_0 $, $\varphi_m $ is holomorphic on $\left\lbrace |z|\leq \frac{\log(m) }{\sqrt{m}} \right\rbrace$, and $||\varphi_{m} ||_{h^m} \leq b_m |z|^{2p'} $ on $U$, moreover
\begin{align}
\lambda_{\left( p_1 ,p_2 ,...,p_n \right)}^{-2} = \int_{\left\lbrace |z|\leq \frac{\log(m) }{\sqrt{m}} \right\rbrace } \left| z_1^{p_1} \cdots z_n^{p_n} \right|^2 a^m dV_g ,
\end{align}
where $dV_g = \left( \frac{\sqrt{-1}}{2}\right)^n \det\left( g_{i\bar{j}}\right) dz_1 \wedge d\zbar_1 \wedge \cdots \wedge dz_n \wedge d\zbar_n $ is the volume form.
\end{lmm}

\begin{proof}
Through the proof of Lemma 1.2 in \cite{tg1}, combined with Proposition \ref{holochart}, this lemma can be proved.
\end{proof}

So the estimation of Tian's Peak sections is reduced to the estimation of $\lambda_{\left( p_1 ,p_2 ,...,p_n \right)}$. Now we need to control some special Fourier coefficients on distinguish boundaries of polydiscs.

\begin{lmm}
\label{ode}
Let $f$ be a smooth real function on a domain that contains the closed polydisc $\bar{U}=\bar{D}(0,\delta )^n $, where $\delta >0$ be a constant.

$(i).$ If $k\in\mathbb{Z}_+ $, $f(0)= \nabla f(0)=0$, $|f|\leq K_1 $ for some $K_1 >0$, and $\left|\partial\partialbar f\right| \leq K_2 $ for some $K_2 >0$, then we have a constant $C>0$ which depends only on $\delta $, $n$, $k$, $K_1$ and $K_2$, such that
$$\int_{\prod_{j=1}^{n}\partial D(0,r_j )} f(z) \cos k\theta_1 d\theta_1 \wedge\cdots\wedge d\theta_n \leq Cr^{2} |\log(r)|\delta_{k,2} + Cr^2 ,$$
when $r=\sqrt{\sum_{j=1}^{n} r_j^2 }  < \frac{\delta }{2n} $, where $z_j = r_j e^{\theta_j \sqrt{-1}}$, and $\delta_{i,j}$ is the Kronecker symbol. In addition, when $k=0 $, we can assume that $C=C(n,K_2 )$.

$(ii).$ If $k\in\mathbb{Z}_+ $, $f(0)= \nabla f(0)= \nabla^2 f(0) =\nabla^3 f(0)=0$, $|f| \leq K_1 $ for some $K_1 >0$, and $\left|\frac{\partial^4 f}{\partial z_i\partial z_j\partial \zbar_s \partial \zbar_t} \right| \leq K_2 $ for some $K_2 >0$, $\forall i,j,s,t$, then we have a constant $C>0$ which depends only on $\delta $, $n$, $k$, and $K$, such that
$$\int_{\prod_{j=1}^{n}\partial D(0,r_j )} f(z) \cos k\theta_1 d\theta_1 \wedge\cdots\wedge d\theta_n \leq C \left( \delta_{k,2} +\delta_{k,4} \right) r^{4} |\log(r)| +Cr^4 ,$$
when $r=\sqrt{\sum_{j=1}^{n} r_j^2 }  < \frac{\delta }{2n} $, where $z_j = r_j e^{\theta_j \sqrt{-1}}$, and $\delta_{i,j}$ is the Kronecker symbol. In addition, when $k=0 $, we can assume that $C=C(n,K_2 )$.
\end{lmm}

\begin{proof}
See Appendix \ref{pode}.
\end{proof}

Assume that $m>\max \left\lbrace m_0 , \frac{ |\log(r)|^2}{r^2} \right\rbrace $ from now. We begin to estimate $\lambda^{-2}_{(p_1 ,\cdots ,p_n )}$ on $M$. 
\begin{lmm}
\label{ldlm}
Under the notations and assumptions of Lemma \ref{peaksec} with the additional condition $|\sec |\leq K$ on $B_1 (x_0 )$, we have the following estimates:
\begin{eqnarray}
\left| \frac{m^{n+p}}{P!}\lambda^{-2}_{(p_1 ,\cdots ,p_n )} -\frac{ 1 }{ \pi^{p} } \right| & \leq & Cr^{-2}m^{-1} , \label{ld1}
\end{eqnarray}
where $C=C\left( Q,n,p,K,\alpha ,C_0 \right) >0$ is a constant.
\end{lmm}

\begin{proof}

Recall the definition $$\lambda_{(p_1 ,\cdots ,p_n)}^{-2} = \int_{\left\lbrace |z|\leq \frac{ \log(m)}{\sqrt{m}} \right\rbrace }  \left| z_1^{p_1} \cdots z_n^{p_n} \right|^2 a^m  dV_g ,$$
we have
\begin{eqnarray*}
& & \bigg| \int_{\left\lbrace |z|\leq \frac{\log(m) }{\sqrt{m}} \right\rbrace }  \left| z_1^{p_1} \cdots z_n^{p_n} \right|^2 a^m  dV_g \\
& & - \int_{\left\lbrace |z|\leq \frac{\log(m)}{\sqrt{m}} \right\rbrace } \left| z_1^{p_1} \cdots z_n^{p_n} \right|^2 e^{-\pi m|z|^2 } dV_{\mathbb{C}^n } \bigg| \\
& \leq &  \int_{\left\lbrace |z|\leq \frac{\log(m) }{\sqrt{m}} \right\rbrace }  \left| z_1^{p_1} \cdots z_n^{p_n} \right|^2 \left| a^m  -e^{-\pi m|z|^2 } \right| \left| \det\left( g_{i\bar{j}}\right) -1 \right| dV_{\mathbb{C}^n } \\
& & + \left| \int_{\left\lbrace |z|\leq \frac{\log(m) }{\sqrt{m}} \right\rbrace }  \left| z_1^{p_1} \cdots z_n^{p_n} \right|^2 a^m  -\left| z_1^{p_1} \cdots z_n^{p_n} \right|^2 e^{-\pi m|z|^2 } dV_{\mathbb{C}^n } \right| \\
& & + \left| \int_{\left\lbrace |z|\leq \frac{\log(m) }{\sqrt{m}} \right\rbrace }  \left| z_1^{p_1} \cdots z_n^{p_n} \right|^2 \left[ \det\left( g_{i\bar{j}}\right) -1  \right] e^{-\pi m|z|^2 }  dV_{\mathbb{C}^n } \right| .
\end{eqnarray*}

Proposition \ref{holochart} shows that:
\begin{eqnarray*}
& & \int_{\left\lbrace |z|\leq \frac{\log(m) }{\sqrt{m}} \right\rbrace }  \left| z_1^{p_1} \cdots z_n^{p_n} \right|^2  \left| a^m  -e^{-\pi m|z|^2 } \right| \left| \det\left( g_{i\bar{j}}\right) -1 \right| dV_{\mathbb{C}^n } \\ 
& \leq & C_1 r^{-4-2\alpha } m\int_{\left\lbrace |z|\leq \frac{\log(m) }{\sqrt{m}} \right\rbrace }  |z|^{2p+4+2\alpha } e^{-\pi m|z|^2 } dV_{\mathbb{C}^n } ,
\end{eqnarray*}
where $C_1 $ is a positive constant depending only on $Q$, $n$, $\alpha $, and $p$.

Since  $$R_{k\bar{k} i\bar{j}} = -\frac{\partial^2 g_{i\bar{j}}}{\partial z_k \partial \zbar_k} - g^{s\bar{t}} \frac{\partial g_{s\bar{j}}}{\partial z_k} \frac{\partial g_{i\bar{t}}}{\partial \zbar_k} ,$$
Lemma \ref{ode} now implies that
\begin{eqnarray*}
& & \left| \int_{\left\lbrace |z|\leq \frac{\log(m) }{\sqrt{m}} \right\rbrace } \left| z_1^{p_1} \cdots z_n^{p_n} \right|^2 \left[ \det\left( g_{i\bar{j}}\right) -1  \right] e^{-\pi m|z|^2 }  dV_{\mathbb{C}^n } \right| \\
& \leq & C_2 r^{-2} \int_{\left\lbrace |z|\leq \frac{\log(m) }{\sqrt{m}} \right\rbrace }  |z|^{2p+2} e^{-\pi m|z|^2 } dV_{\mathbb{C}^n },
\end{eqnarray*}
where $C_2 $ is a positive constant depending only on $Q$, $n$, $p$, and $K$.

By the definition of $a$, we have
$$\frac{1}{2\pi } \frac{\partial^2 }{\partial z_i \partial \zbar_j } \log(a) = -g_{i\bar{j}} ,$$
then we can apply Lemma \ref{ode} to $a-e^{-\pi |z|^2} $. It gives a constant $C_3 $ depending only on $Q$, $n$, $p$, and $K$, such that
\begin{eqnarray*}
& & \left| \int_{\left\lbrace |z|\leq \frac{\log(m) }{\sqrt{m}} \right\rbrace } \left| z_1^{p_1} \cdots z_n^{p_n} \right|^2 \left[ a^m -e^{-\pi m|z|^2 } \right]  dV_{\mathbb{C}^n }  \right| \\
& \leq & mC_3 r^{-2} \int_{\left\lbrace |z|\leq \frac{\log(m) }{\sqrt{m}} \right\rbrace }  |z|^{2p+4} e^{-\pi m|z|^2 } dV_{\mathbb{C}^n } \\
&   & + m^2C_3 r^{-3} \int_{\left\lbrace |z|\leq \frac{\log(m) }{\sqrt{m}} \right\rbrace } |z|^{2p+7} e^{-\pi m|z|^2 }  dV_{\mathbb{C}^n } .
\end{eqnarray*}

Then a straightforward calculation shows that
\begin{eqnarray*}
\int_{\mathbb{C}^n } |z|^{2p+4} e^{-\pi m|z|^2 } + m^{\frac{3}{2}} |z|^{2p+7 } e^{-\pi m|z|^2} dV_{\mathbb{C}^n} & \leq & C_4 m^{-n-p-2} ,
\end{eqnarray*}
where $C_4 =C_4 (n) $ is a constant.

By direct computation, we have
\begin{eqnarray*}
 \left| \int_{\left\lbrace |z|\leq \frac{\log(m) }{\sqrt{m}} \right\rbrace } \left| z^{p_1}_1 \cdots z^{p_n}_n \right|^2 e^{-\pi m|z|^2} dV_{\mathbb{C}^n} - \frac{ P! }{ \pi^{p} m^{n+p}} \right| \leq C_5 m^{-n-p-2}  ,
\end{eqnarray*}
where $C_5 =C_5 (p,n)>0 $ is a constant. We thus get the estimate (\ref{ld1}).
\end{proof}

\begin{rmk}
When $P=(0,\cdots ,0)$, we can replace the condition $|\sec |\leq K$ by $|Ric|\leq K$, and the proof is similar to the above.
\end{rmk}

We now estimate the inner product between peak sections with some other sections.

\begin{lmm}
\label{inprodlmm}
Let $S_m $ be the sections we have constructed in Lemma \ref{peaksec}, and let $T$ be another section of $L^m $ with $\int_M ||T||^2_{h^m} dV_g =1$, which contain no term $z_1^{p_1} \cdots z_n^{p_n} $ in its Taylor expansion at $x_0 $. Then
\begin{eqnarray*}
\left| \int_{M}  \left\langle S_m ,T \right\rangle_{h^m}  dV_g \right| & \leq & C r^{-1-\alpha } m^{-\frac{1+\alpha}{2} } ,
\end{eqnarray*}
where $\left\langle \; ,\; \right\rangle $ is the inner product on the linear space $H^0 (M,L^m )$ induced by the metric $h^m $, and $C=C(Q,n,p ,\alpha )$ is a constant.
\end{lmm}

\begin{proof}
We divide the integral into three parts:
\begin{eqnarray}
 \int_{M}  \left\langle S_m ,T \right\rangle_{h^m}  dV_g & = & \int_{M\big\sq\left\lbrace |z| \leq \frac{\log(m) }{\sqrt{m}} \right\rbrace }  \left\langle S_m ,T \right\rangle_{h^m}  dV_g \nonumber\\
&  & + \left( 1+ \frac{a_m}{m^{2p'}} \right) \lambda_{\left( p_1 ,p_2 ,...,p_n \right) } \int_{\left\lbrace |z| \leq \frac{\log(m) }{\sqrt{m}} \right\rbrace }  \left\langle    \varphi_m   ,T \right\rangle_{h^m}  dV_g \label{inprod}\\
&  & + \left( 1+ \frac{a_m}{m^{2p'}} \right) \lambda_{\left( p_1 ,p_2 ,...,p_n \right)}  \int_{\left\lbrace |z| \leq \frac{\log(m) }{\sqrt{m}} \right\rbrace }  \left\langle    z_1^{p_1} \cdots z_n^{p_n}  e_L^m ,T \right\rangle_{h^m}   dV_g . \nonumber
\end{eqnarray}

Lemma \ref{peaksec} shows that 
\begin{eqnarray*}
& & \left| \int_{M\big\sq\left\lbrace |z| \leq \frac{\log(m) }{\sqrt{m}} \right\rbrace }  \left\langle S_m ,T \right\rangle_{h^m}   dV_g \right| \\ 
& \leq & \left(  \int_{M\big\sq\left\lbrace |z| \leq \frac{\log(m) }{\sqrt{m}} \right\rbrace }  \left\Vert S_m  \right\Vert^2_{h^m}    dV_g  \right)^{\frac{1}{2}} \left(  \int_{M\big\sq\left\lbrace |z| \leq \frac{\log(m) }{\sqrt{m}} \right\rbrace }    \left\Vert T  \right\Vert^2_{h^m}  dV_g  \right)^{\frac{1}{2}} \\
& \leq & C_1 m^{-p' } ,
\end{eqnarray*}
and the similar argument gives
\begin{eqnarray*}
 \left| \left( 1+ \frac{a_m}{m^{2p'}} \right) \lambda_{\left( p_1 ,p_2 ,...,p_n \right) } \int_{\left\lbrace |z| \leq \frac{\log(m)}{\sqrt{m}} \right\rbrace }  \left\langle    \varphi_m   ,T \right\rangle_{h^m}  dV_g \right| & \leq & C_2 m^{-p'} ,
\end{eqnarray*}
where $C_1 $, $C_2 $ are constants depending only on $Q$, $n$, $p$, $\alpha $.

It is sufficient to estimate the last term of (\ref{inprod}) now. 

We assume that $T = f_{T} e_L$ on ${\left\lbrace |z| \leq \frac{\log(m) }{\sqrt{m}} \right\rbrace} $, then $f_T$ is holomorphic on $\left\lbrace |z| \leq \frac{\log(m) }{\sqrt{m}} \right\rbrace $ and contains no term $z^{p_1}_1 \cdots z^{p_n}_n $ in the Taylor expansion at $z=0$. It follows that
$$\int_{ \left\lbrace |z| \leq \frac{\log(m)}{\sqrt{m}} \right\rbrace } z^{p_1}_1 \cdots z^{p_n}_n \bar{f}_T e^{-\pi m|z|^2}  dV_{\mathbb{C}^n} = 0 ,$$
because for each $P\neq 0$, $$\int_{ \prod_{j=1}^{n} \partial D(0,r_j ) } z^{p_1}_1 \cdots z^{p_n}_n d\theta_1 \wedge\cdots\wedge d\theta_{n} =0 .$$

By Schwarz inequality we have
\begin{eqnarray*}
& & \left| \int_{\left\lbrace |z| \leq \frac{\log(m) }{\sqrt{m}} \right\rbrace }  \left\langle    z_1^{p_1} \cdots z_n^{p_n}  e_L^m ,T \right\rangle_{h^m}   dV_g \right|  \\
& \leq &  \left( \int_{\left\lbrace |z| \leq \frac{\log(m) }{\sqrt{m}} \right\rbrace }      |z|^{2p}  \left[ a^m \det\left( g_{i\bar{j}} \right) -e^{-\pi m|z|^2} \right]^{2} a^{-m} \det\left( g_{i\bar{j}} \right)^{-1} dV_{\mathbb{C}^n} \right)^{\frac{1}{2}}           .
\end{eqnarray*}

It follows from Lemma \ref{holochart} that 
\begin{eqnarray*}
& & \left[ a^m \det\left( g_{i\bar{j}} \right) -e^{-\pi m|z|^2} \right]^{2} a^{-m} \det\left( g_{i\bar{j}} \right)^{-1} \\
& \leq & C_{3} \left[ a^m \left( \det\left( g_{i\bar{j}} \right) -1\right) + \left( a^m -e^{-\pi m|z|^2} \right) \right]^{2} a^{-m} \\
& \leq & C_{4} \left[ r^{-2-2\alpha } |z|^{2+2\alpha} e^{- m|z|^2}  + mr^{-1-\alpha }|z|^{3+\alpha } e^{- m|z|^2} \left( 1 -\left( \frac{e^{-\pi |z|^2}  }{a} \right)^m \right) \right] \\
& \leq & C_{5} e^{- m|z|^2} \left( r^{-2-2\alpha } |z|^{2+2\alpha}  + m^2 r^{-2-2\alpha } |z|^{6+2\alpha }  \right),
\end{eqnarray*}
where $C_3 $, $C_4 $, $C_5 $ are positive constants depend only on $Q$, $n$, $p$, $\alpha $.
Then we can conclude that
\begin{eqnarray*}
& & \left| \int_{\left\lbrace |z| \leq \frac{\log(m) }{\sqrt{m}} \right\rbrace }  \left\langle    z_1^{p_1} \cdots z_n^{p_n}  e_L^m ,T \right\rangle_{h^m}   dV_g \right| \\
& \leq &  \left( \int_{\left\lbrace |z| \leq \frac{\log(m) }{\sqrt{m}} \right\rbrace }      |z|^{2p}  \left[ a^m \det\left( g_{i\bar{j}} \right) -e^{-2\pi m|z|^2} \right]^{2} a^{-m} \det\left( g_{i\bar{j}} \right)^{-1} dV_{\mathbb{C}^n} \right)^{\frac{1}{2}}   \\
& \leq & \left( C_{5} r^{-2-2\alpha } \int_{\left\lbrace |z| \leq \frac{\log(m) }{\sqrt{m}} \right\rbrace }      |z|^{2p+2+2\alpha }   e^{- m|z|^2} \left( 1+   m^2  |z|^{4 }  \right) dV_{\mathbb{C}^n}  \right)^{\frac{1}{2}} \\
& \leq & C_6 r^{-1-\alpha } m^{-\frac{n+p+1+\alpha }{2} }
\end{eqnarray*}
for some constant $C_6 =C_6 (Q,n,p,\alpha )>0$.

From (\ref{ld1}), we see that $\lambda_{\left( p_1 ,p_2 ,...,p_n \right) } \leq C_{7} m^\frac{p+n}{2}$ for some constant $C_7 =C_7 (Q,p,n,\alpha ) .$ This gives
\begin{align*}
\left| \int_{M}  \left\langle S_m ,T \right\rangle_{h^m}  dV_g \right|  \leq  C r^{-1-\alpha } m^{-\frac{ 1+\alpha }{2} } ,
\end{align*}
where $C=C(Q,n,p ,\alpha )$ is a constant.
\end{proof}

Now we focus on the adjacent peak sections.

\begin{lmm}
\label{adpeaksec}
Let $S_m $ be the peak sections we have constructed in Lemma \ref{peaksec} for $P=\left( p_1 ,\cdots ,p_n \right) $, and let $T_m $ be the peak sections for $P'=\left( p_1 +k  ,\cdots ,p_n \right) $ for some $k\in\mathbb{N}$. We assume that $|\sec |\leq K$ for some $K>0 $, then we can find a constant $C=C(n,p ,Q,k,K)$, s.t.
$$ \left| \int_{M}  \left\langle S_m ,T_m \right\rangle_{h^m}  dV_g \right|  \leq  Cm^{-1 } +Cm^{-1}\log(m)\left( \delta_{k,2}+\delta_{k,4} \right) , $$
for each $m>m_0 $.
\end{lmm}

\begin{proof}
It is sufficient to show that
$$ \left| \int_{\left\lbrace |z|  \leq \frac{\log(m) }{\sqrt{m}} \right\rbrace}   \left| z_1^{p_1} \cdots z_n^{p_n} \right|^2 z^k_1 a^m  dV_g \right|  \leq  Cm^{-n-p-1-\frac{k}{2} } \left( 1 + \delta_{k,2} \log(m)+\delta_{k,4}\log(m) \right) .$$

Let $\psi = \log(a)+\pi |z|^2 $, $\varphi = \det\left( g_{\alpha\bar{\beta}} \right) -1$. Then we have 
\begin{eqnarray*}
& & \int_{\left\lbrace |z|  \leq \frac{\log(m) }{\sqrt{m}} \right\rbrace}   \left| z_1^{p_1} \cdots z_n^{p_n} \right|^2 z^k_1 a^m  dV_g \\
& = & \int_{\left\lbrace |z|  \leq \frac{\log(m) }{\sqrt{m}} \right\rbrace}   \left| z_1^{p_1} \cdots z_n^{p_n} \right|^2  z^k_1 e^{-\pi m|z|^2} \left( e^{m\psi } -1-m\psi \right) (1+\varphi )dV_{\mathbb{C}^n} \\
&  & + \int_{\left\lbrace |z|  \leq \frac{\log(m) }{\sqrt{m}} \right\rbrace}   \left| z_1^{p_1} \cdots z_n^{p_n} \right|^2  z^k_1 e^{-\pi m|z|^2} \left(  m\psi +\varphi +m\psi\varphi \right) dV_{\mathbb{C}^n} .
\end{eqnarray*}

Since there exists a constant $C_1 =C_1 (n,p ,Q,k,K)$ such that $ |\psi |\leq C_1 |z|^{3 } $, $ |\varphi | \leq C_1 |z| $, and $$\left| e^{m\psi } -1-m\psi \right|\leq C_1 m^2 |z|^{7  } ,$$
we can apply the Lemma \ref{ode} to $\varphi $ and $\psi $ to find a constant $C_2 =C_2 (n,p ,Q,k,K) $ such that 
\begin{eqnarray*}
& & \left| \int_{\left\lbrace |z|  \leq \frac{\log(m) }{\sqrt{m}} \right\rbrace}   \left| z_1^{p_1} \cdots z_n^{p_n} \right|^2 z^k_1 a^m  dV_g \right| \\
& \leq & C_2 \int_{\left\lbrace |z|  \leq \frac{\log(m) }{\sqrt{m}} \right\rbrace}   |z|^{2p+k} e^{-\pi m|z|^2} \left( m|z|^4 +|z|^2 +m|z|^{5} +m^2 |z|^7 \right) dV_{\mathbb{C}^n} \\
& & + C_2 \int_{\left\lbrace |z|  \leq \frac{\log(m) }{\sqrt{m}} \right\rbrace}   |z|^{2p+k} e^{-\pi m|z|^2} \left( m|z|^4 +|z|^2 \right) \left| \log(|z| )\right| \left( \delta_{k,2}+\delta_{k,4} \right) dV_{\mathbb{C}^n}  .
\end{eqnarray*}

Then a straightforward computation gives the following inequality:
\begin{eqnarray*}
\int_{\left\lbrace |z|  \leq \frac{\log(m) }{\sqrt{m}} \right\rbrace}   |z|^{2p+k} e^{-\pi m|z|^2} \left( m|z|^4 +|z|^2 +m|z|^{5} +m^2 |z|^7 \right) dV_{\mathbb{C}^n}
& \leq & C_3 m^{-n-p-1-\frac{k}{2}} ,
\end{eqnarray*}
where $C_3 =C_3 (n,p, k) $ is a constant.

Clearly, $\sqrt{|z|} |\log(|z| )| \geq -2e^{-1} $, and hence
\begin{eqnarray*}
& & \int_{\left\lbrace |z|  \leq \frac{\log(m) }{\sqrt{m}} \right\rbrace}   |z|^{2p+k} e^{-\pi m|z|^2} \left( m|z|^4 +|z|^2 \right) \left| \log(|z| )\right| dV_{\mathbb{C}^n} \\
& \leq & \int_{\left\lbrace m^{-2} \leq |z|  \leq \frac{\log(m) }{\sqrt{m}} \right\rbrace}   |z|^{2p+k} e^{-2\pi m|z|^2} \left( m|z|^4 +|z|^2 \right) \left| \log\left( m^2 \right)\right| dV_{\mathbb{C}^n} \\
& & + \int_{\left\lbrace  |z|  \leq m^{-2}  \right\rbrace}   |z|^{2p+k} \left( m|z|^4 +|z|^2 \right) \left| \log(|z| )\right| dV_{\mathbb{C}^n} \\
& \leq & 2C_{3} m^{-n-p-1-\frac{k}{2}} |\log(m)| + C_4 m^{-2n-4p-2k-3},
\end{eqnarray*}
where $C_4 $ is a constant depends only on $n$, $p$, $Q$, $k$, $K$.
\end{proof}

\section{Pointwise Estimates}
\label{pe}
In this section, we will prove Theorem \ref{thm1}.

Choosing an $L^2$ orthonormal basis $\left\lbrace S^{m}_{j} \right\rbrace_{j=0}^{N_m -1}$ of $H^0 \left( M,L^m \right)$, where $N_m =dim H^0 \left( M,L^m \right)$. Since $L|_{U}$ is a trivial bundle, we can find holomorphic functions $f^m_j\in\mathcal{O}(U)$ s.t. $S^{m}_{j} = f^m_j e_L $ on $U$. By an orthogonal transformation we may further assume that
\begin{eqnarray}
f^m_j (0) & = & 0, \textrm{ for } j\geq 1, \label{fm10}\\
\frac{\partial f^m_j}{\partial z_k} (0) & = & 0, \textrm{ for } j\geq k+1, \; j=1,2,\cdots ,n, \label{fm11} \\
\frac{\partial^2 f^m_j}{\partial z^2_1} (0) & = & 0, \textrm{ for } j\geq n+2. \label{fm12} 
\end{eqnarray}

\begin{lmm}
\label{fm}
Under the conditions stated above, for each given $k\in\mathbb{N}$, $\alpha\in (0,1)$, there exists a positive constant $C$ depends only on $t,\; Q,\; r,\; n,\; k,\; \alpha ,\; K  $, such that
\begin{eqnarray}
\left| \sqrt{\frac{1}{ m^n}} \left| f^m_0 (0) \right| -1\right| & \leq & Cm^{-1} ,\label{fm0}\\
\left| \sqrt{\frac{1}{ m^{n+1}}} \left| \frac{\partial f^m_1 }{\partial z_1} (0) \right| -\sqrt{\pi } \right| & \leq & Cm^{-1} ,\label{fm1} \\
\left| \sqrt{\frac{1}{2 m^{n+2}}} \left| \frac{\partial f^m_{n+1} }{\partial z^2_1} (0) \right| -\pi \right| & \leq & Cm^{-1} ,\label{fm2} \\
m^{-\frac{n+k-2}{2}} \left| \frac{\partial f^m_0 }{\partial z^k_1} (0) \right| + m^{-\frac{n+k-1}{2}} \left| \frac{\partial^2 f^m_1 }{\partial z^{1+k}_1} (0) \right| & \leq & C +C\log(m)\left( \delta_{k,2}+\delta_{k,4} \right) ,\label{fm3}
\end{eqnarray}
and
\begin{eqnarray}
 m^{-\frac{n+\alpha }{2}} \left| \frac{\partial^2 f^m_j }{\partial z^2_1} (0) \right|  \leq  C, \;  1\leq j\leq n . \label{fm4}
\end{eqnarray}
\end{lmm}

\begin{proof}
Let $T_0 $, $T_1 $, $\cdots$, $T_{n+1}$ be peak sections of $L^m $ for $P=(0,\cdots ,0)$,  $(1,0,\cdots ,0)$, $\cdots$, $(0,\cdots ,1)$, and $(2,0,\cdots ,0)$, respectively. We can find constants $\beta_{ij} $, satisfying that $T_i = \sum_{j=0}^{N_m -1}\beta_{ij} S^m_j $, for $j=1,2,\cdots ,n+1$. By Lemma \ref{inprodlmm}, 
\begin{eqnarray*}
& & \left| \int_{M}  \left\langle T_0 ,\sum_{j=1}^{N_m -1}\beta_{0j} S^m_j \right\rangle_{h^m}  dV_g \right| \\
 & \leq & C_{1} m^{-\frac{1+\alpha}{2} } \left| \int_{M}  \left\langle \sum_{j=1}^{N_m -1}\beta_{0j} S^m_j ,\sum_{j=1}^{N_m -1}\beta_{0j} S^m_j \right\rangle_{h^m}  dV_g \right|^{\frac{1}{2}} \\
& \leq & C_{1} m^{-\frac{1+\alpha}{2} } ,
\end{eqnarray*}
and hence
\begin{eqnarray*}
\sum_{j=1}^{N_m -1}\left| \beta_{0j} \right|^2 & \leq & C_{1} m^{-\frac{1+\alpha}{2} } ,\\ 
\end{eqnarray*}
where $C_1 =C_1 (t,Q,r,n,p,\alpha )$ is a constant. Since $\int_M ||T_0 ||^2_{h^m} dV_g =1$, it follows that
\begin{eqnarray*}
\sum_{j=1}^{N_m -1}\left| \beta_{0j} \right|^2 & = & \left| \int_{M}  \left\langle T_0 ,\sum_{j=1}^{N_m -1}\beta_{0j} S^m_j \right\rangle_{h^m}  dV_g \right|  \\
& \leq &  C_1 m^{-\frac{1+\alpha}{2}} \left( \sum_{j=1}^{N_m -1}\left| \beta_{0j} \right|^2 \right)  \\
& \leq & C^2_1 m^{-1-\alpha }     . 
\end{eqnarray*}

Thus we have 
\begin{eqnarray*}
 \left| \sqrt{\frac{1}{ m^n}} \left| f^m_0 (0) \right| -1\right| & \leq & \left| \sqrt{\frac{1}{ m^n}} \left| \beta_{00} \right|^{-1} \lambda_{(0,\cdots ,0)} -1\right| +C_2 m^{-1 } \\
& \leq & C_3 m^{-1 },
\end{eqnarray*}
where $C_2 =C_2 (t,Q,r,n,p,\alpha ,K)$, $C_3 =C_3 (t,Q,r,n,p,\alpha ,K)$ are constants.

Writing $T_i = f_{T_i} e_L $ locally, then we have $f_{T_i} (0)=0 $, $\forall i>0$, and $ \frac{\partial f_{T_i}}{ \partial z_j } (0)=0 $, when $i\neq j$. Then we have $\beta_{ij} =0$, if $i>j$. By a similar argument, Lemma \ref{inprodlmm} now shows that
\begin{eqnarray*}
\sum_{j=i}^{N_m -1}\left| \beta_{ij} \right|^2 & = & 1 , \\
\sum_{j=i+1}^{N_m -1}\left| \beta_{ij} \right|^2 & \leq & C_4 m^{-1-\alpha } ,
\end{eqnarray*}
and hence $$ \left| \left| \beta_{ii} \right|^2 -1 \right| \leq  C_{4} m^{-1-\alpha } ,$$ for some constant $C_4 =C_4 (t,Q,r,n,p,\alpha )$.

It follows that there exists a constant $C_5 =C_5 (t,Q,r,n,p,\alpha ,K ) >0$ such that
\begin{eqnarray*}
\left| \sqrt{\frac{1}{ m^{n+1}}} \left| \frac{\partial f^m_1 }{\partial z_1} (0) \right| - \sqrt{\pi} \right| & \leq &  C_5 m^{-1} , \\
\left| \sqrt{\frac{1}{2 m^{n+2}}} \left| \frac{\partial f^m_{n+1} }{\partial z^2_1} (0) \right| -\pi \right| & \leq & C_5 m^{-1} .
\end{eqnarray*}

Next we come to (\ref{fm3}) and (\ref{fm4}). 

Let $B=(\beta_{ij} )_{0\leq i,j\leq n+1}$ be a matrix, and let $X=(f^{m}_{i} )_{0\leq i\leq n+1}$, $Y=(f_{T_i} )_{0\leq i\leq n+1}$ be row vectors with function elements. Then $T_i = \sum_{j=0}^{N_m -1}\beta_{ij} S^m_j $ gives $$\frac{\partial^2 Y}{\partial z_1^2}=\frac{\partial^2 X}{\partial z_1^2}B ,\quad \frac{\partial Y}{\partial z_j}=\frac{\partial X}{\partial z_j }B ,\; 1\leq j\leq n ,$$ 
and thus $$\frac{\partial^2 X}{\partial z_1^2}=\frac{\partial^2 Y}{\partial z_1^2}B^{-1} ,\quad \frac{\partial X}{\partial z_j}=\frac{\partial Y}{\partial z_j }B^{-1} ,\; 1\leq j\leq n .$$ 

Since $f_{T_i} (0)=0 $, $\forall i>0$, and $ \frac{\partial f_{T_i}}{ \partial z_j } (0) =0 $, when $i\neq j$, then $\sum_{j=i+1}^{N_m -1}\left| \beta_{ij} \right|^2  \leq  C_4 m^{-1-\alpha }$ shows that $|| B^{-1} -I_{n+2} ||\leq (1+C_4 )^n(2+n)^n m^{-\frac{1+\alpha }{2}} $, where $I_{n+2} $ is the identity matrix. Then Lemma \ref{ldlm} implies (\ref{fm4}).

Apply the Lemma \ref{adpeaksec} to the peak sections $T_0 $ and $T_1 $, one can see that 
\begin{eqnarray*}
\left| \beta_{01} \bar{\beta}_{11} +\sum_{j=2}^{N_m -1} \beta_{0j} \bar{\beta}_{1j} \right| & \leq & C_6 m^{-1} ,
\end{eqnarray*}
where $C_6 =C_6 (t,Q,r,n,p ,K ) $ is a constant, and the Cauchy-Schwartz inequality shows that 
\begin{eqnarray*}
\left| \sum_{j=2}^{N_m -1} \beta_{0j} \bar{\beta}_{1j} \right| & \leq & C_4 m^{-1} ,
\end{eqnarray*}
hence we have $ \left| \beta_{01} \bar{\beta}_{11} \right|  \leq  (C_4 + C_6 ) m^{-1} $, and $ \left| \beta_{01} \right|  \leq  (1+C_4)(C_4 + C_6 ) m^{-1} $.

Recall the definition of peak section, $ \frac{\partial f_{T_0}}{\partial z_1 } (0)=0$, and we can rewrite it as $$ \beta_{00} \frac{\partial f^m_0}{\partial z_1} (0) + \beta_{01} \frac{\partial f^m_1}{\partial z_1} (0) =0 . $$

By the argument above, 
\begin{eqnarray*}
\left| \frac{\partial f^m_0}{\partial z_1} (0) \right| & = & \left|  \beta_{00}^{-1}  \beta_{01} \frac{\partial f^m_1}{\partial z_1} (0)  \right| \\
& \leq & (1+C_4 )^2 (C_4 +C_6)m^{-1} \cdot (1+C_5 ) m^{\frac{n+1}{2}} .
\end{eqnarray*}

Similarly, we can find a constant $C_7 =C_7 (t,Q,r,n,p ,K ) $ satisfying that $ |\beta_{1,n+1} |\leq C_7 m^{-1} $, and 
\begin{eqnarray*}
\left| \frac{\partial f^m_1}{\partial z^2_1} (0) \right| & \leq & \left|  \beta_{11}^{-1}  \beta_{1,n+1} \frac{\partial f^m_{n+1} }{\partial z^2_1} (0)  \right| +C_7 m^{-\frac{3 }{4} }\cdot m^\frac{n+1}{2}  \\
& \leq & (1+C_4 )\cdot (1+ C_7 ) m^{-1} \cdot (1+2C_5 ) m^{\frac{n+2}{2}} .
\end{eqnarray*}
It gives (\ref{fm3}) when $k=1 $. Likewise, the Lemma \ref{adpeaksec} can obtain (\ref{fm3}) in general, if we consider the peak sections of $L^m$ for $P=(1,0,\cdots , 0),\cdots ,(k,0,\cdots , 0) $. The proof is almost the same as that of the case $k=1$, so we omit it.
\end{proof}

Now we are ready to prove Theorem \ref{thm1}.

\noindent \textbf{Proof of Theorem \ref{thm1}: } The proof is completed by showing the following inequalities: 
\begin{eqnarray*}
\left| \frac{1}{\pi m} \frac{\partial^2 }{\partial z_1 \partial \zbar_1} \log\left( \sum_{j=0}^{N_m -1} \left| f_j^m \right|^2 \right) -1 \right| (x_0) \leq Cm^{-1} ,
\end{eqnarray*}
and
\begin{eqnarray*}
\left| \frac{1}{m} \frac{\partial^3 }{\partial z_1^2 \partial \zbar_1} \log\left( \sum_{j=0}^{N_m -1} \left| f_j^m \right|^2 \right) \right| (x_0) \leq Cm^{-\frac{1}{2}} .
\end{eqnarray*}

Since $f^m_j$ are holomorphic functions satisfying the assumptions (\ref{fm10}) and (\ref{fm11}), we can conclude that
\begin{eqnarray*}
 \frac{\partial^2 }{\partial z_1 \partial \zbar_1} \log\left( \sum_{j=0}^{N_m -1} \left| f_j^m \right|^2 \right) (x_0 ) & = &  \frac{\partial}{\partial z_1 } \left(  \frac{ \sum_{j=0}^{N_m -1}  f^m_j \partial \bar{f}^m_j / \partial\zbar_{1} }{ \sum_{j=0}^{N_m -1} \left| f_j^m \right|^2 }   \right) (x_0 ) \\
& = &  \frac{ \left| \partial f^m_1 /\partial z_1 \right|^2  }{\left| f^m_0 \right|^2} (x_0 ) ,
\end{eqnarray*}
and similarly,
\begin{eqnarray*}
& & \frac{\partial^3  }{\partial z_1^2 \partial \zbar_1} \log\left( \sum_{j=0}^{N_m -1} \left| f_j^m \right|^2 \right) (x_0 ) \\
& = & \frac{ \partial^2 f^m_1 /\partial z^2_1  \cdot \partial \bar{f}^m_1 / \partial\zbar_1 }{ \left| f_0^m \right|^2 } - \frac{2   \bar{f}^m_0 \partial {f}^m_0 / \partial z_{1} \cdot \partial f^m_1 /\partial z_1 \cdot \partial \bar{f}^m_1 / \partial\zbar_1   }{  \left| f_0^m \right|^4 } .
\end{eqnarray*}

Combining Corollary \ref{secvol} with the lemmas in this section, we can assert that
\begin{eqnarray*}
\left| \frac{ \left| \partial f^m_1 /\partial z_1 \right|^2  }{2\pi m\left| f^m_0 \right|^2} -1 \right| & \leq &  \left| \frac{ (m+n+1)!/m!  }{m\cdot (m+n)!/m!} -1 \right| +C_1 m^{-1} \\
& = &  \frac{ n+1  }{m} +C_1 m^{-1} ,
\end{eqnarray*}
where $C_1 =C_1 (K,v ) >0$ is a constant.

Likewise, we can find a constant $C_2 =C_2 (K,v ) >0 $ satisfies that
\begin{eqnarray*}
\left| \frac{\partial^3  }{\partial z_1^2 \partial \zbar_1} \log\left( \sum_{j=0}^{N_m -1} \left| f_j^m \right|^2 \right) (x_0 ) \right| \leq C_2 m^{ \frac{1}{2}} .
\end{eqnarray*}

This theorem follows. \qed

\section{H\"older Estimates}
\label{lhestm}
In this section, we will show that the sequence of Bergman metrics, $\{ g_{m} \}_{m=m_0 }^{\infty} $, converges to $g$ in the $C^{1,\alpha}$-topology, and we will try to control the $C^2 $-norm and $W^{2,p}$-norm of $g_{m}$ now. Although $g_m $ cannot converge into $g$ in the Sobolev space $W^{2,1} $ (see section \ref{exp}), we can still make some estimates.

Fix a uniform holomorphic chart $(z_1 ,\cdots ,z_n )$ as in the previous section, then we follow the assumptions and notations here.

\begin{lmm}
\label{pppp}
Under the conditions stated in the previous section, there exists a positive constant $C$, $m_0 $ depends only on $t$, $Q$, $r$, $n$, $K$, such that
\begin{eqnarray*}
\left| \frac{\partial^2 g_{k\bar{l} ,m} }{\partial z_i \partial z_j} \right| \leq C\log(m) , 
\end{eqnarray*}
when $m>m_0 $,  where $g_{k\bar{l} ,m} = g_m \left( \frac{\partial }{\partial z_k} ,\frac{\partial }{\partial \zbar_l }  \right) $.
\end{lmm}

\begin{proof}
It's sufficient to show that 
\begin{eqnarray*}
\left| \frac{\partial^4 }{\partial z^3_1 \partial \zbar_1} \log\left( \sum_{j=0}^{N_m -1} \left| f_j^m \right|^2 \right) (x_0 ) \right| \leq Cm\log(m) . 
\end{eqnarray*}

By a straightforward computation, we obtain
\begin{eqnarray*}
& & \frac{\partial^4 }{\partial z^3_1 \partial \zbar_1} \log\left( \sum_{j=0}^{N_m -1} \left| f_j^m \right|^2 \right) (x_0 ) \\
& = & \frac{ \partial^3 f^m_1 /\partial z^3_1 \cdot \partial \bar{f}^m_1 /\partial \zbar_1  }{\left| f^m_0 \right|^2} - 3\frac{ \sum_{i=1}^{2}  \sum_{j=0}^{1} \bar{f}^m_0 \partial^i f^m_0 /\partial^i z_1 \cdot \partial^{2-i} f^m_j /\partial z^{2-i}_1 \cdot \partial \bar{f}^m_j /\partial \zbar_1   }{\left| f^m_0 \right|^4} \\
& & + 6\frac{| \partial f^m_0 /\partial z_1 |^2 \cdot \left( \sum_{j=0}^{1} | \partial f^m_j /\partial z_1 |^2 \right) }{\left| f^m_0 \right|^4} + 6\frac{ (  f^m_0  )^2 ( \partial f^m_0 /\partial z_1 )^2 \cdot  | \partial f^m_1 /\partial z_1 |^2  }{\left| f^m_0 \right|^6} .
\end{eqnarray*}
Then this lemma follows from Lemma \ref{fm}.
\end{proof}

Similarly, we have:

\begin{lmm}
\label{pppb}
Under the conditions stated in the previous section, there exists a positive constant $C$, $m_0 $ depends only on $t$, $Q$, $r$, $n$, $K$, such that
\begin{eqnarray*}
\left| \frac{\partial^2 g_{k\bar{l} ,m} }{\partial z_i \partial \zbar_j} \right| \leq C , 
\end{eqnarray*}
when $m>m_0 $,  where $g_{k\bar{l} ,m} = g_m \left( \frac{\partial }{\partial z_k} ,\frac{\partial }{\partial \zbar_l }  \right) $.
\end{lmm}

\begin{proof}
It's sufficient to show that 
\begin{eqnarray*}
\left| \frac{\partial^4 }{\partial z^2_1 \partial \zbar^2_1} \log\left( \sum_{j=0}^{N_m -1} \left| f_j^m \right|^2 \right) (x_0 ) \right| \leq Cm . 
\end{eqnarray*}

By a straightforward computation, we obtain
\begin{eqnarray*}
& & \frac{\partial^4 }{\partial z^2_1 \partial \zbar^2_1} \log\left( \sum_{j=0}^{N_m -1} \left| f_j^m \right|^2 \right) (x_0 ) \\
& = & \frac{\sum_{i=1}^{n+1} \left| \partial^2 f^m_i /\partial z^2_1 \right|^2 }{\left| f^m_0 \right|^2} + \frac{ -2 \left| \partial f^m_1 /\partial z_1 \right|^4 + 4 \left| \partial f^m_0 /\partial z_1 \right|^2 \left| \partial f^m_1 /\partial z_1 \right|^2 }{\left| f^m_0 \right|^4} \\
& &-4Re\left(   \frac{\partial^2 f^m_1 /\partial z^2_1 \cdot \partial \bar{f}^m_1 /\partial \zbar_1 \cdot f^m_0 \cdot \partial \bar{f}^m_0 /\partial \zbar_1  }{\left| f^m_0 \right|^4}     \right) .
\end{eqnarray*}
By Lemma \ref{fm}, 
\begin{eqnarray*}
\frac{\sum_{i=1}^{n} \left| \partial^2 f^m_i /\partial z^2_1 \right|^2 }{\left| f^m_0 \right|^2} \leq C_1 \frac{m^{n+\frac{1}{2}}}{m^n } = C_1 m^{\frac{1}{2}} ,
\end{eqnarray*}
\begin{eqnarray*}
 \frac{\left| \partial f^m_0 /\partial z_1 \right|^2 \left| \partial f^m_1 /\partial z_1 \right|^2 }{\left| f^m_0 \right|^4} \leq C_1 \frac{ m^{\frac{n+1}{2}} \cdot m^{\frac{n-1}{2}} }{ m^{2n} } = C_1 ,
\end{eqnarray*}
and
\begin{eqnarray*}
 \left|   \frac{\partial^2 f^m_1 /\partial z^2_1 \cdot \partial \bar{f}^m_1 /\partial \zbar_1 \cdot f^m_0 \cdot \partial \bar{f}^m_0 /\partial \zbar_1  }{\left| f^m_0 \right|^4}     \right| 
 \leq  C_1 \frac{ m^{\frac{n}{2}} \cdot m^{\frac{n+1}{2}} \cdot  m^{\frac{n}{2}} \cdot m^{\frac{n-1}{2}} }{ m^{2n} } = C_1 ,
\end{eqnarray*}
where $C_1 =C_1 (t,Q,r,n,K)$ is a constant.

Apply Lemma \ref{fm} again, we can conclude that 
\begin{eqnarray*}
& & \left| \frac{ \left| \partial^2 f^m_{n+1} /\partial z^2_1 \right|^2 }{\left| f^m_0 \right|^2} - \frac{ 2 \left| \partial f^m_1 /\partial z_1 \right|^4 }{\left| f^m_0 \right|^4} \right| \\
& \leq & \frac{2m^{2n+2} }{m^{2n}} (1+ C_2 m^{-1} ) - \frac{2m^{2n+2} }{m^{2n}} (1- C_2 m^{-1} ) \\
& = & 4C_2 m ,
\end{eqnarray*}
where $C_2 =C_2 (t,Q,r,n,K)$ is a constant.

This proves the Lemma.
\end{proof}

Using the classical $L^p$-estimates of elliptic partial differential equations, we can give the following result as a corollary:
\begin{coro}
\label{lpestcoro}
Let $(M,g)$ be a polarized K\"ahler manifold. Assume that there are constants $K,v >0$ and $t\geq 0$ s.t. $Ric\geq -t g $ on $M$, $\left| \sec \right|\leq K$ on $B_1 (x_0 )$, and $\Vol B_{1}(x_0 ) >v ,$ for $ x_0 \in M$. Then for each $q>1$, we have constants $\epsilon =\epsilon (K,v,t)$, $m_0 =m_0 (K,v,t )\in\mathbb{N}$ and $C=C(K,v,t,q ) >0$, such that
$$ ||\nabla^2 g_m ||_{q, B(x_0 ,\epsilon )} \leq C ,\; \forall m>m_0 ,$$
where $||\cdot ||_{q, B(x_0 ,\epsilon )}$ is the $L^q$ norm of tensors on $B(x_0 ,\epsilon )$ which is induced by $g$.
\end{coro}

\begin{proof}
It follows the $L^p$-estimates about the solutions of elliptic partial differential equations. For example, see Theorem 9.13 in \cite{gilt1}.
\end{proof}

\begin{rmk}
Actually, we can obtain a $C^{1,\alpha}$ estimate by combining Corollary \ref{lpestcoro} and Lemma \ref{thm1}. But this estimate is weaker than the estimate we will make below. 
\end{rmk}

Now we need some estimates of Newtonian potential here. Let $\Gamma (x)$ be the fundamental solution of Laplace's equation on $\mathbb{R}^n$, i.e.,
\begin{displaymath}
\Gamma (x) =\Gamma (|x| ) = \left\{ \begin{array}{ll}
\frac{1}{n(2-n)|B_1| } |x|^{2-n}, & \textrm{if $n>2$,}\\
\frac{1}{2\pi} \log|x|, & \textrm{if $n=2$,}\\
\end{array} \right.
\end{displaymath}
where $|B_1 |$ is the volume of unit ball in $\mathbb{R}^n$. Fix $R>0$. For each $f\in L^{\infty} \left( B_R (0) \right) $, let $\omega $ be the Newtonian potential of $f$, then $\omega (x) =\int_{B_R (0)} \Gamma (x-y) f(y) dy$, $ \forall x\in \mathbb{R}^n $. The classical theory of Newtonian potential gives $D_i \omega (x) =\int_{B_R (0)} D_i \Gamma (x-y) f(y) dy$, where $D_i = \frac{\partial}{\partial x_i}$. By the definition of $\Gamma $, we have $D_i \Gamma (x-y) = \frac{1 }{n|B_1 |  } \cdot \frac{ x_i -y_i}{ |x-y|^n }$. Then we will estimate the difference of of $\nabla \omega $ locally.

\begin{lmm}
\label{newtpot}
There are constant $ C=C(n ) $, $\epsilon =\epsilon (n,R  )$, such that
$$  \left| \nabla \omega (x) -\nabla \omega (0) \right|  \leq C||f||_{L^{\infty}} \cdot |x| \left| \log\left( |x| \right) \right|  ,$$
when $|x|\leq \epsilon $, where $||\cdot ||_{L^{\infty}} $ is the $L^\infty $-norm.
\end{lmm}

\begin{proof}
Since $D_i \Gamma (x) = \frac{1 }{n|B_1 |  } \cdot \frac{ x_i }{ |x|^n }$, $ i=1,\cdots ,n$, a direct computation shows that $$ \left| \nabla^2 \Gamma (x) \right|\leq \frac{2n^2 }{|B_1|} |x|^{-n} .$$

Let $r \leq \min\left\lbrace \frac{1}{10R}, \frac{R}{10} \right\rbrace $. If $|x|\leq \frac{r}{5}$, we have
\begin{eqnarray*}
D_i \omega (x) - D_i \omega (0) & = & \int_{B_{r} (0)} \left( D_i \Gamma (x-y) - D_i\Gamma (-y) \right) f(y) dy \\
& & + \int_{B_{R} (0) -B_{r} (0)} \left( D_i \Gamma (x-y) - D_i\Gamma (-y) \right) f(y) dy .
\end{eqnarray*} 

When $|y|\geq r$, we can find $t\in (0,1)$, such that $$\left| D_i \Gamma (x-y) - D_i\Gamma (-y) \right| \leq |x|\cdot\left| \nabla^2 \Gamma (tx-y) \right| .$$ 

Without loss of generality, we can assume that $||f||_{L^{\infty}} =1$. Then 
\begin{eqnarray*}
& & \left|\int_{B_{R} (0) -B_{r} (0)} \left( D_i \Gamma (x-y) - D_i\Gamma (-y) \right) f(y) dy \right| \\  
& \leq &  |x| \int_{B_{R} (0) -B_{r} (0)} \left|\nabla^2 \Gamma (tx-y) \right| dy \\
& = & 2^{n+1}n^3 |x| \left( \log(R) -\log(r)  \right) \leq 2^{n+2}n^3 |x| \left| \log(r) \right| .
\end{eqnarray*} 

On the small ball $B_r (0)$, we can obtain another estimate as follows. 
\begin{eqnarray*}
\left|\int_{B_{r} (0)} \left( D_i \Gamma (x-y) - D_i\Gamma (-y) \right) f(y) dy \right| & \leq &  2\int_{B_{2r} (0)} \left| D_i\Gamma (-y)\right|  dy \\
& \leq &  2\int_{B_{2r} (0)}  \frac{|y|^{1-n}}{n|B_1|}  dy = 4r.
\end{eqnarray*} 

Combine the estimation of this two integrals, we have
\begin{eqnarray*}
 \left| \nabla \omega (x) -\nabla \omega (0) \right|  & \leq & 4 r + 2^{n+1}n^3 |x| \left| \log(r) \right| .
\end{eqnarray*}

Fix $\epsilon = \min\left\lbrace \frac{1}{200R^2 }, \frac{R^2 }{200} \right\rbrace $. If $|x|\leq \epsilon $, then we can choose $r=-|x|\log(|x|)$. It is easy to check that $r \leq \min\left\lbrace \frac{1}{10R}, \frac{R}{10} \right\rbrace $, $r>5|x|$, and thus
\begin{eqnarray*}
 \left| \nabla \omega (x) -\nabla \omega (0) \right|  & \leq & -4 |x|\log(|x|) + 2^{n+1}n^3 |x| \left| \log(r) \right| \\
 & \leq &  2^{n+3}n^3 |x|\log(|x|) .
\end{eqnarray*}

This completes the proof.
\end{proof}

Since $\Delta \omega =f$, we can obtain a similar estimate about the solution of Poisson's equation as following:
\begin{lmm}
\label{poissonc1alpha}
If $f\in C^\infty \left( B_R(0) \right)$, and $|f|+|\Delta f|\leq K$, then we can find constant $ C=C(n ,K ) $, $\epsilon =\epsilon (n,R,K  )$, such that for each $|x|\leq \epsilon $,
$$  \left| \nabla f (x) -\nabla f (0) \right|  \leq C |x| \left| \log\left( |x| \right) \right|  .$$ 
\end{lmm}

\begin{proof}
Let $\psi $ be the Newtonian potential of $\Delta f$, then we have $\Delta (\psi -f) =0$, and it is immediate that there exists a constant $C_1 =C_1 (n,R,K)>0$ such that $|\psi |\leq C_1 $.

By the interior derivative estimates for harmonic functions, there is a constant $C_2 =C_2 (n,R,K) ,$ satisfies that $|\nabla^2 (\psi -f) |\leq C_2 $ on $B_{\frac{R}{2}} (0)$. By Lemma \ref{newtpot}, $\left| \nabla \psi (x) -\nabla \psi (0) \right|  \leq C_3 |x| \left| \log\left( |x| \right) \right|  $, when $|x|\leq \epsilon $, where $\epsilon $ and $C_3 $ are constants depend only on $n$, $R$, $K$.

Hence we have
\begin{eqnarray*}
 \left| \nabla f (x) -\nabla f (0) \right|  & \leq & \left| \nabla \psi (x) -\nabla \psi (0) \right| +  \left| \nabla (\psi -f) (x) -\nabla (\psi -f) (0) \right| \\
 & \leq &  C_3 |x| \left| \log\left( |x| \right) \right| + C_2 |x| ,
\end{eqnarray*}
and the lemma follows.
\end{proof}

Now we will consider the $C^{1,\alpha}$-convergence of Bergman metrics. Recall that $$ R_{k\bar{k} i\bar{j}} = -\frac{\partial^2 g_{i\bar{j}}}{\partial z_k \partial \zbar_k} - g^{s\bar{t}} \frac{\partial g_{s\bar{j}}}{\partial z_k} \frac{\partial g_{i\bar{t}}}{\partial \zbar_k} ,$$
then we can apply Lemma \ref{poissonc1alpha} to $g_{i\bar{j}}$:
\begin{lmm}
\label{oric1a}
Under the conditions stated in the previous section, there exists a positive constant $\epsilon $, $C$, $m_0 $ depends only on $t$, $Q$, $r$, $n$, $K$, such that for each $|x-y|\leq \epsilon $,
\begin{eqnarray*}
\left| \nabla g_{i\bar{j} } (x) -\nabla g_{i\bar{j} } (y) \right|  \leq C |x-y| \left| \log\left( |x-y| \right) \right|  , 
\end{eqnarray*}
when $m>m_0 $,  where $g_{i\bar{j} } = g \left( \frac{\partial }{\partial z_i} ,\frac{\partial }{\partial \zbar_j }  \right) $.
\end{lmm}

Now we are ready to prove Theorem \ref{thm2}.

\noindent \textbf{Proof of Theorem \ref{thm2}: }
It is sufficient to show that 
\begin{eqnarray*}
\left| \nabla g_{i\bar{j} ,m } (z) -\nabla g_{i\bar{j} ,m } (0) - \nabla g_{i\bar{j} } (z) + \nabla g_{i\bar{j} } (0) \right|  \leq C |z|^{\alpha } m^{\frac{-1+\alpha }{2}}  \left| \log(m) \right|^{\alpha}   , 
\end{eqnarray*}
on the fixed uniform holomorphic chart $(z_1 ,\cdots ,z_n )$. 

Clearly, Theorem \ref{thm1} implies that
\begin{eqnarray*}
& & \left| \nabla g_{i\bar{j} ,m } (z) -\nabla g_{i\bar{j} ,m } (0) - \nabla g_{i\bar{j} } (z) + \nabla g_{i\bar{j} } (0) \right| \\
& \leq &  \left| \nabla g_{i\bar{j} ,m } (z)  - \nabla g_{i\bar{j} } (z) \right| + \left|  \nabla g_{i\bar{j} ,m } (0) - \nabla g_{i\bar{j} } (0) \right|  \\
& \leq & C_1 m^{-\frac{1}{2}}  , 
\end{eqnarray*}
where $C_1 = C_1 (K,v,t)$ is a constant.

By Lemma \ref{pppp}, Lemma \ref{pppb} and Lemma \ref{oric1a}, we can conclude that there are constants $\epsilon = \epsilon (K,v,t) $ and $C_2 = C_2 (K,v,t) $, such that if $|z| \leq \epsilon $
\begin{eqnarray*}
& & \left| \nabla g_{i\bar{j} ,m } (z) -\nabla g_{i\bar{j} ,m } (0) - \nabla g_{i\bar{j} } (z) + \nabla g_{i\bar{j} } (0) \right| \\
& \leq &  \left| \nabla g_{i\bar{j} ,m } (z)  - \nabla g_{i\bar{j} ,m } (0) \right| + \left|  \nabla g_{i\bar{j} } (z) - \nabla g_{i\bar{j} } (0) \right|  \\
& \leq & C_2 |z|   \log(m) +C_2 |z|   \left| \log\left( |z|  \right) \right|  . 
\end{eqnarray*}

Let $\delta = m^{-\frac{1}{2}} \left| \log(m) \right|^{-1} $. If $|z|  \geq \delta$, an easy computation shows that
\begin{eqnarray*}
& & \left| \nabla g_{i\bar{j} ,m } (z) -\nabla g_{i\bar{j} ,m } (0) - \nabla g_{i\bar{j} } (z) + \nabla g_{i\bar{j} } (0) \right| \\
& \leq &  |z|^{\alpha} m^{\frac{\alpha}{2}} \left| \log(m) \right|^{\alpha} \cdot C_1 m^{-\frac{1}{2}} \\
& \leq &  C_1 |z|^{\alpha } m^{\frac{-1+\alpha }{2}}  \left| \log(m) \right|^{\alpha} .
\end{eqnarray*}
Since $ \lim_{m\to\infty} m^{-\frac{1}{2}} \left| \log(m) \right|^{-1} =0 $, we can find $m_1 = m_1 (\epsilon ) >0 $ such that $\delta\leq\epsilon$, $\forall m>m_1 $. By above argument, if $|z|  \leq \delta $ and $m>m_1 $, we have
\begin{eqnarray*}
& & \left| \nabla g_{i\bar{j} ,m } (z) -\nabla g_{i\bar{j} ,m } (0) - \nabla g_{i\bar{j} } (z) + \nabla g_{i\bar{j} } (0) \right| \\
& \leq & C_2 |z|^{\alpha} \delta^{1-\alpha} \log(m) +C_2 |z|^{\alpha} |z|^{1-\alpha} \left| \log\left( |z|  \right) \right|  \\
& \leq & C_2 |z|^{\alpha } \left( m^{\frac{-1+\alpha }{2}}  \left| \log(m) \right|^{\alpha} + |z|^{1-\alpha} \left| \log\left( |z|  \right) \right| \right) .
\end{eqnarray*}

A direct computation shows that 
\begin{eqnarray*}
\delta^{1-\alpha} \left| \log\left( \delta \right) \right| \leq 2 m^{-\frac{1-\alpha }{2}} \log(m) .
\end{eqnarray*}
We thus get
\begin{eqnarray*}
& & \left| \nabla g_{i\bar{j} ,m } (z) -\nabla g_{i\bar{j} ,m } (0) - \nabla g_{i\bar{j} } (z) + \nabla g_{i\bar{j} } (0) \right| \\
& \leq & C_2 |z|^{\alpha } \left( m^{\frac{-1+\alpha }{2}}  \left| \log(m) \right|^{\alpha} + |z|^{1-\alpha} \left| \log\left( |z|  \right) \right| \right) \\
& \leq & 2C_2 |z|^{\alpha }  m^{\frac{-1+\alpha }{2}}  \left| \log(m) \right|^{\alpha} ,
\end{eqnarray*}
when $ |z| \leq\delta $, and $m>m_1 +m_2 $, which proves the theorem. \qed

\section{Examples}
\label{exp}
Examples \ref{swc}-\ref{hgs} show that we cannot control the $C^0 $-$convergence$ rate of Bergman metrics if we drop only one of the conditions $\sec\geq K$, $\sec\leq K$ or $\Vol B_1 ( x_0 ) >v$ in Theorem \ref{thm1}.

Example \ref{sharp} provides a $C^{1,1}$ polarized pointed K\"ahler manifold $\left( M,g,L,x_0 \right)$ that satisfies $$ \liminf_{m\to\infty } \sqrt{m} \left\Vert \nabla g_{m} \right\Vert > 0 .$$

Example \ref{oscillation} demonstrates that the conditions in Theorem \ref{thm1} are not sufficient to control the convergence rate of $\nabla^2 g_m $ in $L^1$ norm.

\begin{exap}
\label{swc}
\upshape Let $M=\mathbb{C}P^1$, $L=\mathcal{O} (1)$, $\theta\in (0,1)$, then we have two open sets 
\begin{eqnarray*}
U_0 & = & \left\lbrace \left[ 1 , w \right]\in \mathbb{C}P^1  \right\rbrace ,\\
U_1 & = & \left\lbrace \left[ z , 1 \right]\in \mathbb{C}P^1 : \left| z \right| <1 \right\rbrace
\end{eqnarray*} 
in $M$, such that $U_0 \cup U_1 =M $. Choose a radial cut-off function $\eta \in C_0^{\infty} \left( B_1(0) \right) \subset C_0^{\infty } \left( \mathbb{C} \right)$ such that $0\leq\eta\leq 1$ and $\eta =1$ on $B_{\frac{1}{2}} (0)$. Then we construct a sequence of $C^{1,1}$ functions $f_n \in C^{1,1} \left( \mathbb{C} \right) $,
\begin{displaymath}
f_n (z) = \left\{ \begin{array}{ll}
\frac{1}{2}(\theta -1)\left( e^{2n} |z|^2 -1-2n \right), & \textrm{if $|z|\leq e^{-n} $,}\\
\eta(|z|)(\theta -1)\log(|z|) - \left( 1-\eta(|z|) \right) \log\left( 1+|z|^2\right),  & \textrm{if $|z|> e^{-n} $.}
\end{array} \right.
\end{displaymath}

Clearly, $\Delta f_n|_{\left\lbrace |z|<e^{-n} \right\rbrace} \leq 0$. Then we can find a sequence of functions $u_{n} =u_n (|z|) \in C^{\infty} (\mathbb{C} ) $ such that $\left| u_n - f_n \right| \leq 1 $, $u_n = f_n$ on $\mathbb{C} - B_{e^{-n}} (0)$, and $\Delta u_n  \leq 0 $ on $B_{e^{-n}} (0)$. Then there is a K\"ahler metric $g_{n}'  $ on $M$ such that $g'_{n,1\bar{1}} = e^{-2u_n } $ on $U_0 $, and $g'_{n,1\bar{1}} = \frac{1}{\left( 1+|z|^2\right)^2} $ on $U_1 $. Let $g_n =\frac{g_{n}'}{\int_M \omega_{g_{n}' }}$, it's clear that $\omega_{g_n} \in c_{1} (L) $, and there exists a constant $K>0$ such that $\sec>-K$, $\forall n$.

For each given $m\in\mathbb{N}$, there is a basis $\left\lbrace z_0^j z_1^{m-j} \right\rbrace_{j=0}^m $ of $H^0 \left( M,L^m \right) $. Since $f_n (z)\leq (\theta - 1)\log(|z|)$ on $B_{\frac{1}{2}} (0)$ and $|z|^{1-\theta } \in L^{2} \left( B_1 (0) \right) $, we have a constant $c >0 $ satisfying $c < \int_M \omega_{g_{n}' } <\frac{1}{c}$, $\forall n\in\mathbb{N} $. Let $h_n=\frac{e^{\varphi_n}}{\left| z_0 \right|^2 +\left| z_1 \right|^2 }$ be the unique hermitian metric on $L$ such that $$Ric(h_n)=\omega_{g_n} = -\frac{\sqrt{-1}}{2\pi } \partial\partialbar \varphi_n +\omega_{FS} ,$$ and $\int_{M} \varphi_n \omega_{FS} =0 $. Apply the boundedness of the Green operator to $\left( M,\omega_{FS} \right)$, we have a constant $t>0$ such that $$\left\Vert \varphi_{n} \right\Vert_{L^2 (M)} = \left\Vert G_{FS}\Delta_{FS} \varphi_{n} \right\Vert_{L^2 (M)} \leq t\left\Vert \Delta_{FS} \varphi_{n} \right\Vert_{L^2 (M)}  ,\; \forall n\in\mathbb{N} .$$ 

Now we have a constant $C>0$ such that
$$ \sup_{M} \left| \varphi_n \right| \leq \sqrt{C} \left( \left\Vert \omega_g -\omega_{FS} \right\Vert_{L^{2} (M)} + \left\Vert \varphi_{n} \right\Vert_{L^{2} (M)} \right) \leq C ,\; \forall n\in\mathbb{N} .$$
Since $h_n $ is unique ,$\partial\partialbar \varphi_n (e^{ix} w) = \partial\partialbar \varphi_n ( w)$, we have $\varphi_n (w) =\varphi_n (|w|) $ on $U_{0}$, and hence we have
\begin{eqnarray*}
\int_{M} \left\langle z_0^j z_1^{m-j} , z_0^k z_1^{m-k} \right\rangle_{h_n^m} & = & 0, \textrm{ if $j\neq k$,} \\
\int_{M} \left\langle z_0^j z_1^{m-j} , z_0^j z_1^{m-j} \right\rangle_{h_n^m} & = & a_{m,n} .
\end{eqnarray*}

Now we can find a constant $M>0$ such that $ a_{m,n} \in \left( \frac{1}{M} ,M \right) $, $\forall n\in\mathbb{N} $. Then the Bergman metrics $\frac{\sqrt{-1}}{2\pi} \partial\partialbar \log \left( \sum_{j=1}^{m+1} a^{-\frac{1}{2}}_{m,n} |z_{0}|^{2j} |z_1 |^{2m-2j} \right) \leq C' \omega_{FS} $ for some constant $C'>0$. But $$\lim_{n\to \infty} \frac{g_{n,1\bar{1}} ([1,0])}{g_{FS,1\bar{1}} ([1,0]) } =\infty ,$$ it shows that the distances between $g_n$ and the Bergman metrics $g_{n,m}$ satisfying
$$\liminf_{n\to \infty } \sup_{M} \left\Vert g_{n,m } -g_n \right\Vert\geq 1,$$
for each given $m\in\mathbb{N}$.

So we cannot control the rate of convergence of Bergman metrics if we only assume that $\sec\geq -K$ and $\Vol\left( B_1 (x)\right) >v$, $\forall x\in M$.
\end{exap}

\begin{exap}
\upshape Let $M=\mathbb{C}P^1$, $L=\mathcal{O} (1)$, choose an open covering $U_0 ,$ $U_1 $ of $M$, where 
\begin{eqnarray*}
U_0 & = & \left\lbrace \left[ 1 , w \right]\in \mathbb{C}P^1 : \left| w \right| <2 \right\rbrace ,\\
U_1 & = & \left\lbrace \left[ z , 1 \right]\in \mathbb{C}P^1 : \left| z \right| <2 \right\rbrace .
\end{eqnarray*}

Pick a radial cut-off function $\eta = \eta (|z|) \in C_0^{\infty} \left( B_2(0) \right) \subset C_0^{\infty } \left( \mathbb{C} \right)$, s.t. $0\leq\eta\leq 1$ and $\eta =1$ on $B_{\frac{3}{2}} (0)$. For each $n\in\mathbb{N} $, we define $$ f_n (z) = \eta (z) + \left( 1-\eta(z) \right) \eta \left(  \frac{z}{e^{e^n}} \right) \frac{1}{|z|\log(|z|)} + \left( 1-\eta(z) \right)  \left( 1-\eta\left( \frac{z}{e^{e^n}} \right) \right)\frac{1}{e^{n} |z|} .$$

It's obvious that there is a sequence of K\"ahler metrics $g'_n$ on $M$ such that
\begin{displaymath}
g'_{n,1\bar{1} } = \left\{ \begin{array}{ll}
e^{4e^n } f_n^2 \left( e^{2e^n} w \right) & \textrm{on $U_0 $,}\\
e^{4e^n } f_n^2 \left( e^{2e^n} z \right) & \textrm{on $U_1 $.}
\end{array} \right.
\end{displaymath}
 Through a direct calculation, we get $ 2\pi < \Vol\left( M,g'_n \right) <20\pi $, $\forall n\in\mathbb{N} $, and the sectional curvature of $g'_n$ satisfies that $ |\sec |\leq 100\pi $.

Let $g_n =\frac{ 1 }{\sqrt{\Vol(M,g'_n)}} g'_n $, then $\Vol(M,g_n )=1$ shows that the K\"ahler form $\omega_{n}$ associated with $g_n $ belongs to $c_1(L)$, $\forall n\in\mathbb{N}$. Now we can choose a hermitian metric $h_n $ on $L$ for each $n\in\mathbb{N} $, such that $Ric(h_n )=\omega_n $. Since $g_n $ is invariant under the following $S^1$-action:
\begin{eqnarray*}
S^1 \times M & \to & M,\\
\left( e^{\theta \sqrt{-1} } , [z,w]   \right)  & \mapsto &  \left[e^{\theta \sqrt{-1} } z,w \right],
\end{eqnarray*}
$\partial\partialbar \log\left( h_n \right)  = -2\pi \omega_n $ shows that $h_n $ can be represented as radial functions on $U_0 $ and $U_1 $, if we choose the standard trivialization. It follows that $ \left\lbrace z^j w^{m-j} \right\rbrace_{j=0}^{m} $ becomes an $L^2 $-orthogonal basis of $H^{0} \left( M,L^m \right) $, $\forall n\in\mathbb{N}$. Write $ a_{n,m,j} = \left( \int_{M} \left\Vert z^j w^{m-j} \right\Vert^2_{h^m} dV_{g_n}  \right)^{-\frac{1}{2}}  $, then $a_{n,m,j} =a_{n,m,m-j} $, and the $m$-th Bergman metric of $g_n $, which is $g_{n,m}$, satisfies that 
\begin{eqnarray*}
g_{n,m,1\bar{1}} & = & \frac{1}{2\pi } \frac{\partial^2}{\partial z\partial \zbar } \log\left( \sum_{j=0}^{m}  \left| a_{n,m,j} z^{j} \right|^2  \right) \\
& = & \frac{1}{2\pi }  \frac{ \sum_{j=1}^{m}  j^2 \left| a_{n,m,j} z^{j-1} \right|^2 }{ \sum_{j=0}^{m}  \left| a_{n,m,j} z^{j} \right|^2 } - \frac{1}{2\pi } \frac{  \left( \sum_{j=1}^{m}  j\left| a_{n,m,j} z^{2j-1} \right|^2  \right)^2 }{ \left( \sum_{j=0}^{m}  \left| a_{n,m,j} z^{j} \right|^2 \right)^2 }, 
\end{eqnarray*}
on $U_1 $. Setting $z=1$ and $m=2l+1$ one obtains
\begin{eqnarray*}
g_{n,m,1\bar{1}} \left( [1,1] \right) & = & \frac{1}{2\pi }  \frac{ \sum_{j=1}^{m}  j^2 \left| a_{n,m,j} \right|^2 }{ \sum_{j=0}^{m}  \left| a_{n,m,j} \right|^2 } - \frac{1}{2\pi } \frac{  \left( \sum_{j=1}^{m}  j\left| a_{n,m,j} \right|^2  \right)^2 }{ \left( \sum_{j=0}^{m}  \left| a_{n,m,j} \right|^2 \right)^2 } \\
& \geq &  \frac{ m^2 \left| a_{n,m,0} \right|^4 + \sum_{1\leq j < k\leq m} \left(  j-k  \right)^2 \left| a_{n,m,j} \right|^2 \left| a_{n,m,k} \right|^2 }{ 2\pi (m+1)^2 \left( \sum_{j=0}^{m}  \left| a_{n,m,j} \right|^4 \right)  } \\
& \geq & \frac{1}{4\pi (m+1)^2 },
\end{eqnarray*}
when $l>m_0 +2 $. Since $\lim_{n\to \infty } g_{n,1\bar{1}} \left( [1,1] \right) =0 $, we get $$\liminf_{n\to \infty } \sup_{M} \left\Vert g_{n,2m+1 } -g_n \right\Vert\geq 1,$$
for each given $m>m_0 +2$.
\end{exap}

 \begin{exap}
 \label{hgs}
\upshape Let $ U=\mathbb{R}^2 $ be a riemannian manifold with metric $$ g_{U} = dr^2 +\psi^2 (r)d\theta^2 ,$$ where $\psi  $ be a non-negative smooth function on $\mathbb{R} $, such that $ \psi (r) = r $ on $[0,10]$, $\psi >0$ on $ (0,\infty ) $, and $ \psi (r)=e^{-r} $ when $r>20$.

When $k>30 $, we can construct a Riemannian manifold $\left( k\mathbb{T}^2 ,h \right) $ such that $\Vol_h \left( k\mathbb{T}^2  \right) = 10 $, where $k\mathbb{T}^2 = \mathbb{T}^2 \#\cdots \#\mathbb{T}^2 $ ($k$ times). By gluing $B_{\log(k)} (0)\subset U$ and $ k\mathbb{T}^2 $, we can get a Riemannian manifold $(M,g)$ that satisfies $ \Vol_{g} (M)= 400 $, $ \Vol B_1 (p) = \pi $, and $\sec = 0$ on $B_1 (p)$, where $p=0\in U$. Since $dim M=2$, there exist a complex structure on $M$ such that $(M,g)$ becomes a K\"ahler manifold. Let $L=\mathcal{O}_{400p}$. It's clear that $\omega_g \in c_1 (L) $.

For each $m\in\mathbb{N}$, choose $k>30+400m$, then we have $H^0 (M,L^m ) \cong \mathbb{C} $. If not, the Riemann-Hurwitz formula shows that $ -\chi_M =-2deg(f) +\sum_{p\in M} (e_p -1) $ for each meromorphic function on $M$, where $ e_p $ is the ramification index under $f$, and hence $2k-2=-\chi_M \leq\sum_{p\in M} (e_p -1) \leq 400m $.

It shows that if we drop the condition $\sec >-t$, then although the local geometric structure can still be controlled, $m_0 $ may become too large.
\end{exap}

\begin{exap}
\label{sharp}
\upshape Let $M=\mathbb{C}P^1$, $L=\mathcal{O} (1)$, then we consider the dense open subset 
\begin{eqnarray*}
U & = & \left\lbrace \left[ z , 1 \right]\in \mathbb{C}P^1 : z\in\mathbb{C} \right\rbrace
\end{eqnarray*} 
of $M$ and the classical frame $e_L$ of $L$ over $U$. Choose a radial cut-off function $\eta \in C_0^{\infty} \left( B_1(0) \right) $, such that $0\leq\eta\leq 1$, $ \left| \nabla \eta \right| \leq 3 $, $ \left| \nabla^2 \eta \right| \leq 30 $ and $\eta =1$ on $B_{\frac{1}{2}} (0)$. Of course, $$ h \left( e_L ,e_L \right) = \frac{1}{ 1+|z|^2 } e^{\frac{1}{400} {\eta (z)|z|^3 \left( z+\zbar \right)} }  $$ is a $C^{3,1}$ hermitian metric on $L$, and, in consequence, $$ \omega = Ric(h) = \omega_{FS} - \frac{\sqrt{-1}}{800\pi } \partial\partialbar \left(\eta (z) |z|^3 \left( z+\zbar \right) \right) $$
is a $C^{1,1}$ polarized K\"ahler metric with bounded sectional curvature, where $\omega_{FS} $ is the Fubini-Study metric. For each given $k\in\mathbb{Z}_+ $, a trivial verification shows that $$ T^m_k = \left( \int_{M} |z|^{2k} \left\langle e_L ,e_L \right\rangle_{h^m} \omega \right)^{-\frac{1}{2}} z^k e_L $$
are peak sections of the point $ x_0 =[0,1] \in U $. Write $a= h \left( e_L ,e_L \right) $, then we can conclude that
$$ \int_{M} |z|^{2k} \left\langle e_L ,e_L \right\rangle_{h^m} \omega = \frac{k!}{m^{k+1}} +O\left( \frac{1}{m^{k+2}} \right) ,$$
and similarly,
\begin{eqnarray*}
\int_{M} |z|^{2k} z \left\langle e_L ,e_L \right\rangle_{h^m} \omega & = & \frac{4k+1}{800 } \cdot \frac{\sqrt{\pi} (2k+3)!!}{2^{k+3}m^{k+\frac{5}{2}}} +O\left( \frac{1}{m^{k+3} } \right) ,
\end{eqnarray*}
for each given $k\in\mathbb{Z}_+ $, where $O\left( \frac{1}{m^k } \right) $ denotes a quantity dominated by $ \frac{C}{m^k } $ with the constant $C$ depending only on $k$.

Choosing an $L^2$ orthonormal basis $\left\lbrace S^{m}_{j} \right\rbrace_{j=0}^{m}$ of $H^0 \left( M,L^m \right)$ such that the local representation $S^{m}_{j} =f^m_j e_L $ satisfies that $\beta_{ii} >0$, and $$ \frac{\partial^i f^m_j }{\partial z^i} \left( x_0 \right) =0 ,\; \forall  i <j .$$
It is easy to check that there are constants $\beta^m_{ij} $ such that $T^m_i = \sum_{j=0}^{m} \beta^m_{ij} S^m_j $, and the above conditions implies that $ \beta_{ij}=0 $, $\forall i<j $. It follows that $ \beta^m_{ii} = 1+ O\left( \frac{1}{m} \right) $, $ \sum_{j>i} \left| \beta^m_{ij} \right|^2 = O\left( \frac{1}{m^{ \frac{3}{2} } } \right) $ and $ \sum_{j>i+1} \left| \beta^m_{ij} \right|^2 = O\left( \frac{1}{m^2 } \right) $. Consider the $L^2$ inner product of $H^0 \left( M,L^m \right) $, we have
\begin{eqnarray*}
\beta_{01}^m \bar{\beta}^m_{11} + \sum_{j=2}^m \beta_{0j}^m \bar{\beta}^m_{1j} & = & \int_{M} \left\langle T^m_0 ,T^m_1 \right\rangle_{h^m} \omega \\
& = & \frac{3\sqrt{\pi} }{6400m } + O\left( \frac{1}{m^{\frac{3}{2}}} \right) ,
\end{eqnarray*}
hence we can conclude that $$ \beta^m_{01} = \frac{3\sqrt{\pi} }{6400m } + O\left( \frac{1}{m^{\frac{3}{2}}} \right) .$$ 
Similar arguments apply to $T^m_1 $ and $T^m_2 $, one obtains
$$ \beta^m_{12} =  \frac{3\sqrt{\pi} }{512m \sqrt{2} } + O\left( \frac{1}{m^{\frac{3}{2}}} \right) . $$
Since $T^m_0 = \sum_{j=0}^{m} \beta^m_{0j} S^m_j $, we have $$ \beta^m_{00} \frac{\partial f^m_0}{\partial z} \left( x_0 \right) + \beta^m_{01} \frac{\partial f^m_1}{\partial z} \left( x_0 \right) =0 ,$$
hence that $$ \frac{\partial f^m_0}{\partial z} \left( x_0 \right) = - \frac{3\sqrt{\pi} }{6400 } + O\left( \frac{1}{m^{\frac{1}{2}}} \right) .$$
Likewise, $$ \frac{\partial^2 f^m_1}{\partial z^2 } \left( x_0 \right) = - \frac{3\sqrt{\pi} m^{\frac{1}{2}} }{512 } +O(1) .$$
Then we have
\begin{eqnarray*}
\frac{\partial g_{m,1\bar{1}}}{\partial z} \left( x_0 \right) & = & \frac{1}{2\pi m} \frac{ \partial^2 f^m_1 /\partial z^2  \cdot \partial \bar{f}^m_1 / \partial\zbar }{ \left| f_0^m \right|^2 } \\
& & - \frac{1}{\pi m} \frac{ \bar{f}^m_0 \partial {f}^m_0 / \partial z \cdot \partial f^m_1 /\partial z \cdot \partial \bar{f}^m_1 / \partial\zbar   }{  \left| f_0^m \right|^4 } \\
& = & -\frac{369 }{ 128000\sqrt{\pi} m^{\frac{1}{2}} } + O\left( \frac{1}{m} \right) .
\end{eqnarray*}
\end{exap}

\begin{rmk}
Choose a sequence of smooth polarized pointed K\"ahler manifold $\left( M,g_k ,L,x_0 \right)$ such that $ \frac{\partial g_{k,1\bar{1}}}{\partial z} \left( x_0 \right) =0 $, and $ g_k \to g $ in $C^{1,1}$-topology as $k\to\infty $. It is clear that $$ \liminf_{m\to\infty} \lim_{k\to\infty} \sqrt{m} \left| \frac{\partial g_{k,m,1\bar{1}}}{\partial z} \left( x_0 \right) \right| >0 . $$
\end{rmk}

\begin{exap}
\label{oscillation}
\upshape Let $M=\mathbb{C}P^1$, $L=\mathcal{O} (1)$, then we have two open sets 
\begin{eqnarray*}
U_0 & = & \left\lbrace \left[ 1 , w \right]\in \mathbb{C}P^1  \right\rbrace ,\\
U_1 & = & \left\lbrace \left[ z , 1 \right]\in \mathbb{C}P^1 : \left| z \right| <1 \right\rbrace
\end{eqnarray*} 
in $M$, such that $U_0 \cup U_1 =M $. Choose a radial cut-off function $\eta \in C_0^{\infty} \left( B_1(0) \right) \subset C_0^{\infty } \left( \mathbb{C} \right)$, s.t. $0\leq\eta\leq 1$ and $\eta =1$ on $B_{\frac{1}{2}} (0)$. For each $k\in\mathbb{N}$, we define $$\varphi_k = k^{-4} \sin(kz+k\zbar )\sin(\sqrt{-1}kz-\sqrt{-1}k\zbar ) \eta (z) $$ on $U_1 $. Then $h_k = e^{\varphi_k } h_{0} $ gives a Hermitian metric on $L$, where $h_0 $ be the normal metric on $L$, i.e. $h_0 = \frac{1}{|w|^2 +1} $ on $U_0 $, and $h_0 = \frac{1}{|z|^2 +1}$ on $U_1$. Clearly, $Ric(h_0 ) = \omega_{FS} $ on $M$, and hence 
\begin{eqnarray*}
Ric( h_k ) & = & \omega_{FS} - \frac{\sqrt{-1} }{2\pi } \partial\partialbar \varphi_k .
\end{eqnarray*}
For sufficiently large $k$, $Ric\left( h_k  \right) $ is also a K\"ahler form. Let $g_k $ be the K\"ahler metric corresponds to $Ric\left( h_k  \right) $. Recall that $$R_{1\bar{1} 1\bar{1}} = -\frac{\partial^2 g_{1\bar{1}}}{\partial z \partial \zbar } - g^{1\bar{1}} \frac{\partial g_{1\bar{1}}}{\partial z } \frac{\partial g_{1\bar{1}}}{\partial \zbar } ,$$
we can find a constant $C>0 $ such that on $U_1 $, 
\begin{eqnarray}
\left| R_{k, 1\bar{1} 1\bar{1}} + \frac{64 k^4}{\pi }   \varphi_k  - R_{FS,1\bar{1} 1\bar{1}} \right| \leq \frac{C}{k} , \label{rkfs}
\end{eqnarray}
where $ R_{FS,1\bar{1} 1\bar{1}} $ is the curvature of Fubini-Study metric on $U_1 $, and $ R_{k, 1\bar{1} 1\bar{1}} $ is the curvature of $g_k$ on $U_1 $.

It is easy to check that there is a constant $\delta >0$ satisfies that 
\begin{eqnarray}
\liminf_{k\to\infty } \left\Vert  k^4 \varphi_k -f  \right\Vert_{L^1 (M)} >\delta ,\label{oscillationvarphi}
\end{eqnarray}
for each given $f\in L^1 (M)$.

If $\nabla^2 g_{m,k} $ converges to $\nabla^2 g_k $ in $L^1$ norm uniformly as $m\to\infty $, then $ R^{m}_{k, 1\bar{1} 1\bar{1}} $ also converges to $ R_{k, 1\bar{1} 1\bar{1}} $ in $L^1$ sense uniformly, where $ R^m_{k, 1\bar{1} 1\bar{1}} $ is the curvature of $g_{m,k}$ on $U_1 $. But for each fixed $m $, it's clear that the sequence of Bergman metrics $ g_{m,k} $ must be convergence in $C^\infty $-topology as $k\to\infty $. Hence $R^{m}_{k,1\bar{1}1\bar{1} }$ is also convergence for each given $m$, then combining \upshape{(\ref{rkfs})} and \upshape{(\ref{oscillationvarphi})} we deduce that
\begin{eqnarray*}
0 & = & \liminf_{k\to\infty} \left\Vert  R_{k, 1\bar{1} 1\bar{1}} + \frac{64 k^4}{\pi }   \varphi_k  - R_{FS, 1\bar{1} 1\bar{1}}  \right\Vert_{L^1 (M)} \\
& = & \liminf_{m,k\to\infty} \left\Vert  R^{m}_{k, 1\bar{1} 1\bar{1}} + \frac{64 k^4}{\pi }   \varphi_k  - R_{FS, 1\bar{1} 1\bar{1}}  \right\Vert_{L^1 (M)}    \geq \frac{64 \delta }{\pi } ,
\end{eqnarray*}
contradiction. 

If there is a uniform sequence $b_m $ such that $\lim_{m\to\infty} b_m =0$, and 
\begin{eqnarray*}
\left\Vert  m^{1-n}\rho_{\omega_k ,m} -m - a_1 \right\Vert_{L^1 } \leq b_m  ,
\end{eqnarray*}
for any sufficiently large $k$ and $m$, then it is obvious that $a_1 = \frac{scal}{2}$, where $scal$ is the scalar curvature. Since $g_k $ converges to $g$ in $C^\infty $-topology as $k\to\infty $, $m^{1-n}\rho_{\omega_k ,m} -m$ is also convergence in in $C^\infty $-topology as $k\to\infty $, for each given $m$. It follows that the sequence of scalar curvatures of $g_k$ is convergence in $L^1$ norm as $k\to\infty $, contrary to \upshape{(\ref{oscillationvarphi})}.
\end{exap}

 \appendix
 \section{Holomorphic Norms of K\"ahler Manifolds} 
 \label{HNoKM}

The results in this appendix are essentially obtained by M. Anderson, J. Cheeger, and M. Gromov (\cite{ma1},~\cite{majc1},~\cite{jc1},~\cite{mg1}). In fact, this is just a holomorphic version of harmonic norm. The details of classical $C^{m,\alpha}$-$norm$ and $C^{m,\alpha}$-$convergence$ theory can also be found in Chapter 11 of \cite{pp1}. Now we will give some basic properties of holomorphic norms. The following result may be proved in much the same way as Proposition 11.3.2 in \cite{pp1}, where $  \left\Vert \cdot \right\Vert^{holo}_{C^{m,\alpha} ,r} $ is defined in Definition \ref{holonorm}.

\begin{prop}
Given $(M,g,p)$, $m\geq 0$, $\alpha\in (0,1]$ we have:
\begin{eqnarray}
& & \left\Vert (M,g,p) \right\Vert^{holo}_{C^{m,\alpha} ,r} = \left\Vert \left( M,\lambda^{2}g,p \right) \right\Vert^{holo}_{C^{m,\alpha} ,\lambda r } \textrm{ for all }  \lambda >0.\\
& &\textrm{The function } r\mapsto \left\Vert (M,g,p) \right\Vert^{holo}_{C^{m,\alpha} ,r} \textrm{ is increasing, continuous, and converges to $0$ as } r\to 0. \\ 
& &\textrm{If } \left\Vert (M,g,p) \right\Vert^{holo}_{C^{m,\alpha} ,r} < Q,\textrm{ then for all }x_1, x_2\in B_r (0)\textrm{ we have } \\
& & \quad\quad\quad\quad\quad e^{-Q} \min \left\{ |x_1 - x_2 | ,2r-|x_1 |-|x_2 | \right\} \leq |\phi (x_1 ) \phi (x_2 )| \leq e^Q |x_1 -x_2 | .\nonumber\\
& &\textrm{The norm }\left\Vert (M,g,p) \right\Vert^{holo}_{C^{m,\alpha} ,r} \textrm{ is realized by a }C^{m+1,\alpha }\textrm{-chart.}\\
& &\textrm{If $M$ is compact, then }\left\Vert (M,g) \right\Vert^{holo}_{C^{m,\alpha} ,r} = \left\Vert (M,g,p) \right\Vert^{holo}_{C^{m,\alpha} ,r  }\textrm{ for some }p\in M.
\end{eqnarray}
\end{prop}

Now we will introduce the $C^{m,\alpha}$-convergence concept for this norm. The classic definition of $C^{m,\alpha}$-convergence in \cite{pp1} is not directly applicable to K\"ahler geometry, because we have to consider the complex structure on the K\"ahler manifold $M$ at the same time. The following example can explain this observation.

\begin{exap}
Pick a cut-functions $\eta =\eta (|z|)\in C^{\infty}_{0} \left( B_1 (0 )  \right) $, such that $\eta (0)=1$ and $0\leq\eta\leq 1$. Let $\left( M_{i},g_{i},p_{i} \right)$ be the constant sequence $\left( \mathbb{C} ,g ,0 \right)$, where the K\"ahler metric $$g_{1\bar{1}} = 1+{\eta\left( 4\left| z \right|\right) + \eta\left( 6\left| z-\sqrt{-1}\right|\right) +\eta\left( 8\left| z-1-2\sqrt{-1}\right| \right) }. $$

Clearly, $\left( M_{i},g_{i},p_{i} \right)$ converges to $\left( \bar{\mathbb{C}} ,g ,0 \right)$ in the classical pointed $C^{m,\alpha}$-topology, $\forall m\geq 0$, but the pointed K\"ahler manifold $\left( \mathbb{C} ,g ,0 \right)$ doesn't isomorphic to $\left( \bar{\mathbb{C}} ,g ,0 \right)$.
\end{exap}

This example tells us that the definition of $C^{m,\alpha}$-convergence needs to be modified when the manifolds are K\"ahler manifolds.

\begin{definition}[Holomorphic $C^{m,\alpha}$-convergence]
A sequence of pointed complete K\"ahler manifolds is said to converge in the pointed $C^{m,\alpha }$-topology, $\left( M_{i},g_{i},J_{i},p_{i} \right)\to \left( M,g,J,p \right)$, if for every $R>0$ we can find a domain $\Omega\supset B_R (p)\subset M$ and $C^{m+1,\alpha }$-embeddings $F_i :\Omega\to M_i$ for large $i$ such that $F_i (p)=p_i $, $F_i (\Omega )\supset B_R \left( p_i \right) $, $F_{i}^* g_i \to g$ and $F_{i}^* J_i \to J$ in the $C^{m,\alpha  }$-topology.
\end{definition}

We can now state and prove the Arzela-Ascoli type theorem on the K\"ahler manifolds.

\begin{thm}
For given $Q>0$, $n\geq 1$, $m\geq 0$, $\alpha\in (0,1]$, and $r>0$ consider the class $\mclass$ of complete, pointed K\"ahler n-manifolds $(M,g,p)$ with $\left\Vert (M,g) \right\Vert^{holo}_{C^{m,\alpha} ,r} \leq Q$. The class $\mclass $ is compact in the pointed $C^{m,\beta }$-topology for all $\beta < \alpha$.
\end{thm}

\begin{proof}
This theorem can be proved by modifying the proof of Theorem 11.3.6 in \cite{pp1}, since the limit of a sequence of holomorphic functions is also holomorphic.
\end{proof}

\begin{rmk}
We can also define the holomorphic $W^{m,p }$-norm of K\"ahler manifolds. When $ mp>n $, we can obtain similar propositions about $W^{m,p }$-norm. See \cite{pp2} for details.
\end{rmk}

Then we will prove a result about the continuity of holomorphic $C^{m,\alpha }$-norm.

\begin{prop}
Suppose $\left( M_i ,g_i ,J_i ,p_i \right)\to (M,g,J,p)$ in $C^{m,\alpha}$, $m\geq 1$, $\alpha >0$, and  the Ricci curvature $Ric(g_i )\geq -c\omega_{g_i} $, for some constant $c>0$ independent of $i$. Then $$\left\Vert \left( M_i ,g_i ,p_i \right) \right\Vert^{holo}_{C^{m,\alpha} ,r} \to\left\Vert (M,g,p) \right\Vert^{holo}_{C^{m,\alpha} ,r} \textrm{for all $r>0$.}$$

Moreover, when all the manifolds have uniformly bounded diameter
$$\left\Vert \left( M_i ,g_i  \right) \right\Vert^{holo}_{C^{m,\alpha} ,r} \to\left\Vert (M,g) \right\Vert^{holo}_{C^{m,\alpha} ,r} \textrm{for all $r>0$.}$$
\end{prop}

\begin{proof}
First, we show the easy part: $$\liminf_{i\to\infty } \left\Vert \left( M_i ,g_i ,p_i \right) \right\Vert^{holo}_{C^{m,\alpha} ,r} \geq\left\Vert (M,g,p) \right\Vert^{holo}_{C^{m,\alpha} ,r}.$$

For each $Q>\liminf \left\Vert \left( M_i ,g_i ,p_i \right) \right\Vert^{holo}_{C^{m,\alpha} ,r}$, we can find holomorphic charts $\phi_{i} : B_r (0) \to M_{i}$ with the requisite properties for sufficiently large $i>0$. After passing to a subsequence, we can make these charts converge to a holomorphic chart
$$\phi=\lim F^{-1}_{i}\circ \phi_{i} : B_r (0)\to M.$$
Since the metrics and complex structure converge in $C^{m,\alpha}$, $\phi$ must be a holomorphic chart, and it follows that $\left\Vert (M,g,p) \right\Vert^{holo}_{C^{m,\alpha} ,r} \leq Q.$

For the reverse inequality
$$\limsup \left\Vert \left( M_i ,g_i ,p_i \right) \right\Vert^{holo}_{C^{m,\alpha} ,r} \leq\left\Vert (M,g,p) \right\Vert^{holo}_{C^{m,\alpha} ,r},$$
select $Q>\left\Vert (M,g,p) \right\Vert^{holo}_{C^{m,\alpha} ,r} .$ Then we can find $\epsilon >0$ such that also $\left\Vert (M,g,p) \right\Vert^{holo}_{C^{m,\alpha} ,r+2\epsilon } <Q.$ Choose $\varphi :B_{r+3\epsilon } (0 )\to U\in M$ satisfying the usual conditions, and let 

\begin{eqnarray*}
U_{i} & = & F_{i} \left( \varphi \left( B_{ r+\epsilon } (0 ) \right) \right), \\
V_{i} & = & F_{i} \left( \varphi \left( B_{ r+\epsilon } (0 ) \right) \right).
\end{eqnarray*}

Obviously, we can assume that $U_{i}$, $V_{i}$ are domains with smooth boundaries 
\begin{eqnarray*}
\partial U_{i} & = & F_{i} \left( \varphi \left( \partial B_{ r+\epsilon } (0 ) \right) \right) , \\
\partial V_{i} & = & F_{i} \left( \varphi \left( \partial B_{ r } (0 ) \right) \right).
\end{eqnarray*}
for each $i\in\mathbb{N}$.

Consider the coordinates $z_{1},...,z_{n}$ of $B_{r+3\epsilon} (0 )\subset \mathbb{C}^{n}$, and denote the function $ z_k\circ\varphi^{-1}\circ F^{-1}_{i}$ on $F_{i} \left( \varphi \left( B_{r+2\epsilon} (0 ) \right) \right) $ as $z_{i,k}$. 

Since $2\partialbar z_{i,k} = \left( J_{i}+\sqrt{-1} \right) dz_{i,k}$ and $\partialbar^{2} =0$ on $M_i$, we can conclude that
\begin{eqnarray*}
2\partial\partialbar \left( \sum_{k=1}^{n} \left| z_{i,k}\right|^{2} \right) & = & \sum_{k=1}^{n} d\left[ z_{i,k} \left( J_{i}+\sqrt{-1} \right) d\zbar_{i,k} +\zbar_{i,k} \left( J_{i}+\sqrt{-1} \right) dz_{i,k} \right] \\
& = &\sum_{k=1}^{n}  dz_{i,k} \wedge  \left( J_{i}+\sqrt{-1} \right) d\zbar_{i,k} + \sum_{k=1}^{n} z_{i,k}  d \left(  J_{i} d\zbar_{i,k} \right)\\
& & +\sum_{k=1}^{n}  d\zbar_{i,k} \wedge  \left( J_{i}+\sqrt{-1} \right) dz_{i,k} + \sum_{k=1}^{n} \zbar_{i,k}  d \left(  J_{i} dz_{i,k} \right).
\end{eqnarray*}

Note that the above formula only contains the first-order derivative of $J_{i}$ at most. By the definition of holomorphic $C^{m,\alpha }$-convergence, we have \begin{eqnarray}\label{eqn1}
lim_{i\to\infty}\left\Vert \varphi^{*}\left[ F^{*}_{i} \partial\partialbar \left( \sum_{k=1}^{n} \left| z_{i,k}\right|^{2} \right) \right] -\sum_{k=1}^{n} dz_{k}\wedge d\zbar_{k} \right\Vert_{C^{m-1,\alpha} \left( B_{r+2\epsilon} (0) \right) } =0 .
\end{eqnarray}

It shows that $U_{i}$ must be Stein manifold for sufficiently large $i>0$. We now plan to use the $L^2$ method to find coordinates $w_{i,k}$ on $U_{i}$ that are close to $z_{i,k}$.

Apply an argument similar as above, we can obtain
$$\lim_{i\to\infty}\left\Vert  \partialbar z_{i,k}   \right\Vert_{C^{m,\alpha} (U_{i})} =0 .$$

Combining $Ric(g_i )\geq -c\omega_{g_i} $ with (\ref{eqn1}), we can assert that $$10e^{Q} \partial\partialbar \left( \sum_{k=1}^{n} \left| z_{i,k}\right|^{2} \right) -Ric(g_{i}) \geq \omega_{g_{i}} ,$$
for sufficiently large $i>0$.

Pick weight function $\psi_i = 10e^{Q} \sum_{k=1}^{n} \left| z_{i,k}\right|^{2} $ on such $U_{i}$. Proposition \ref{l2m} now gives a $C^{m+1,\alpha }$ function $u_{i,k}$ such that $\partialbar (u_{i,k} - z_{i,k})=0$, and $$\int_{U_{i}} |u_{i,k} |^{2} e^{-\psi_i } dV_{g_i} \leq \int_{U_i} |\partialbar z_{i,k} |^{2} e^{-\psi_i} dV_{g_i} .$$ 

Let $w_{i,k}= z_{i,k} - u_{i,k} $. The proof is completed by showing that $$\lim_{i\to\infty}\left\Vert u_{i,k} \circ F_i \right\Vert_{C^{m+1,\alpha} (B(0,r) )}\to 0 .$$

Using the K\"ahler conditions on $\left( M_i ,g_i ,J_i \right)$, we can conclude that
\begin{eqnarray*}
\lim_{i\to\infty}\left\Vert  \Delta_{g_i} z_{i,k}   \right\Vert_{C^{m-1,\alpha} (U_{i})} & \leq &  \lim_{i\to\infty}\left\Vert  \partialbar z_{i,k}   \right\Vert_{C^{m,\alpha} (U_{i})} =0.
\end{eqnarray*}

By Moser iteration, $\lim_{i\to\infty}\left\Vert z_{i,k}   \right\Vert_{C^{0} (V_{i})} = 0$. According to the Schauder's elliptic estimate, we have
\begin{eqnarray*}
\lim_{i\to\infty}\left\Vert  z_{i,k}   \right\Vert_{C^{m+1,\alpha} (V_{i})} & \leq & \lim_{i\to\infty}C_{1} \left\Vert  \Delta_{g_i} z_{i,k}   \right\Vert_{C^{m-1,\alpha} (V_{i})}
+ \lim_{i\to\infty}C_{2} \left\Vert   z_{i,k}   \right\Vert_{C^{0} (V_{i})} =0,
\end{eqnarray*}
where $C_{1}$, $C_{2}$ are constants independent of $i$, which completes the proof.
\end{proof}

As an application, we will construct $C^{1,\alpha} $-bounded holomorphic chart here.

\begin{lmm}
\label{c1alphaest}
Let $\alpha\in (0,1)$, $n\geq 1$, $K>0$ and $R>0$ be given. For every $Q>0$, there is an $r=r(n,\alpha ,K,R )>0$ such that if the K\"ahler $n$-manifold $(M,g)$ satisfies
\begin{eqnarray*}
\sup_{B_{1} (x) } |Ric| &\leq & K ,\\
inj(p) & \geq & R,\; \forall p\in B_{1} (x),
\end{eqnarray*}
then $\left\Vert (M,g,x) \right\Vert^{holo}_{C^{1,\alpha} ,r}  \leq Q .$
\end{lmm}

\begin{proof}
Relacing harmonic norms by holomorphic norms in Lemma 2.2 in \cite{ma1}, we can obtain the proof of this lemma.
\end{proof}

Then we state the Cheeger's lemma.

\begin{lmm}[Cheeger, 1967]
Given $n\geq 2$, $v,K>0$, and an $n$-dimensional Riemannian manifold $(M,g)$ with
\begin{eqnarray*}
 \sup_{B_{1} (x) } |\sec | & \leq & K  ,\\
 \Vol B_1 (x) & \geq & v ,
\end{eqnarray*}
then $inj (x)\geq R$, where $R$ depends only on $n$, $K$, and $v$.
\end{lmm}
\begin{proof}
The proof can be found on page 34 of \cite{jc1}.
\end{proof}

Combining those two lemmas above:
\begin{coro} 
\label{secvol}
Given $n\geq 1$ and $\alpha\in (0,1)$, $v,K>0$, one can find $r(n,\alpha ,K ,R)>0$ for each $Q>0$ such that if the K\"ahler $n$-manifold $(M,g)$ satisfies
\begin{eqnarray*}
\sup_{B_{1} (x) } |\sec | &\leq & K ,\\
\Vol B_1 (x) & \geq & v,
\end{eqnarray*}
then $\left\Vert (M,g,x) \right\Vert^{holo}_{C^{1,\alpha} ,r}  \leq Q .$
\end{coro}

\section{Proof of Lemma \ref{ode}}
\label{pode}

We consider the case $n=1$ at first. In this case, we have $z=re^{\theta\sqrt{-1}}$, and $U=D_\delta (0 )$. Let $h(r) = \int_{\partial D_r (0)} f\cos k\theta d\theta $. Then $\Delta = \frac{\partial^2}{\partial r^2} +\frac{1}{r} \frac{\partial}{\partial r} + \frac{1}{r^2} \frac{\partial^2}{\partial\theta^2} $ shows that
\begin{eqnarray}
h'' + \frac{1}{r} h' - \frac{k^2 h}{r^2} & = & \int_{\partial D_r (0)} \Delta f \cos k\theta d\theta - \frac{1}{r^2} \int_{\partial D_r (0)} \frac{\partial^2 f}{\partial\theta^2} \cos k\theta d\theta - \frac{k^2 h}{r^2} \label{laplace} \\
& = & \int_{\partial D_r (0)} \Delta f \cos k\theta d\theta  . \nonumber
\end{eqnarray}
Likewise, 
\begin{eqnarray}
h'''' + \frac{2}{r} h''' - \frac{2k^2 +1}{r^2} h'' +\frac{2k^2 +1}{r^3} h' +\frac{k^2 ( k^2 -4 )}{r^4} h  = \int_{\partial D_r (0)} \Delta^2 f \cos k\theta d\theta . \label{doblaplace} 
\end{eqnarray}
Define $\varphi (t) = h\left( e^t \right)$, we can rewrite (\ref{laplace}) as
$$ \varphi'' - k^2 \varphi = e^{2t} \int_{\partial D_{e^t} \left( 0 \right)} \Delta f \cos k\theta d\theta ,$$
and hence $$ \left| \varphi'' - k^2 \varphi \right|\leq 16n^2 Ke^{2t} .$$ 

Then $ \left( e^{kt} \varphi' - ke^{kt} \varphi \right)' = e^{kt} \left( \varphi'' - k^2 \varphi \right) $ shows that we can find $a\in \mathbb{R}$ such that 
$$ \left| \varphi' - k \varphi -ae^{-kt} \right|\leq 16n^2 Ke^{2t} , $$
when $t\leq \log \left( \delta \right) $. Similarly, we can find $b\in\mathbb{R} $ such that 
$$ \left| \varphi -\frac{a (1-\delta_{k,0} ) }{k+\delta_{k,0} } e^{-kt} -a\delta_{k,0} t -be^{kt} \right|\leq 16n^2 Ke^{2t} + 16n^2 Kte^{2t} \delta_{k,2}  , $$
when $t\leq \log \left( \delta \right) $.

Assume that $f(0)= \nabla f(0) =0$ now, then we have $a=b=0$ when $k\neq 2$. If $k=2$, we can obtain $a=0 $, and hence $$|b\delta^2 - h(\delta ) |\leq 16n^2 K\delta^2 + 16n^2 K\delta^2 |\log(\delta )| .$$ It follows that $$|h|\leq 16n^2 C_1 r^{2} + 16n^2 Kr^{2} \log(r) \delta_{k,2}  ,\; \forall r\in [0,\delta ] ,$$
where $C_1 =C_1 (n,K,\delta )>0 $ be a constant. This gives $(i)$.

Now we will prove $(ii)$ in the case $n=1$. Rewrite (\ref{doblaplace}) as
$$ \varphi'''' -4\varphi''' -2(k^2 -2)\varphi'' +4k^2 \varphi' +k^2(k^2 -4)\varphi = e^{4t} \int_{\partial D_{e^t} \left( 0 \right)} \Delta^2 f \cos k\theta d\theta .$$

The classical Schauder's estimates shows that $|f|+|\nabla f|+|\nabla^2 f| +|\nabla^3 f| \leq C_2 $ on $\prod_{j=1}^{n} D\left( 0,\frac{\delta }{2n} \right) $, where $C_2 $ is a constant depending on $n$, $K$, and $\delta $. Recall that $\varphi (t) = h\left( e^t \right)$, we can rewrite (\ref{doblaplace}) as
$$ \varphi'''' - 4\varphi''' -2\left( k^2-2 \right) \varphi'' +4k^2 \varphi' +k^2 \left( k^2 -4 \right) \varphi = e^{4t} \int_{\partial D_{e^t} \left( 0 \right)} \Delta^2 f \cos k\theta d\theta ,$$
then we can conclude that $$ \left| \varphi'''' - 4\varphi''' +2\varphi'' +4\varphi' -\varphi \right|\leq 16n^4 Ke^{4t} .$$ 

Argument similar to above implies that
$$ |\varphi (t) | \leq C_3 e^{4t} +C_3 |t|e^{4t} \left( \delta_{k,2} +\delta_{k,4} \right) ,$$
when $t\leq \log(\delta ) -\log(2n) $. It gives $(ii)$ in the case $n=1$. 

When $n>1$, a similar argument gives
\begin{eqnarray*}
& & \left| \int_{\prod_{j=1}^{n}\partial D_{r_j } (0 )} f(z) \cos k\theta_1 d\theta_1 \wedge\cdots\wedge d\theta_n \right| \\
& \leq & 2\pi \left|  \int_{\prod_{j=1}^{n-1}\partial D_{r_j } (0 )} f\left(  z' ,0\right) \cos k\theta_1 d\theta_1 \wedge\cdots\wedge d\theta_{n-1} \right| \\
& & + (2\pi )^{n+2} C_1 r_1^2 ,
\end{eqnarray*}
where $z=\left( z' ,z_n \right) \in U$. 

By induction on $n$, we can reduce $(i)$ to the special case $n=1$, which we have proved.

Apply the same argument again, we can conclude that
 \begin{eqnarray*}
& & \left| \int_{\prod_{j=1}^{n}\partial D_{r_j } (0 )} f(z) \cos k\theta_1 d\theta_1 \wedge\cdots\wedge d\theta_n \right| \\
& \leq & 2\pi \left|  \int_{\prod_{j=1}^{n-1}\partial D_{r_j } (0 )} f\left(  z' ,0\right) \cos k\theta_1 d\theta_1 \wedge\cdots\wedge d\theta_{n-1} \right| \\
& & + 2\pi r^2_n \left|  \int_{\prod_{j=1}^{n-1}\partial D_{r_j } (0 )} \frac{\partial^2 f}{\partial z_n \partial \zbar_n}\left(  z' ,0\right) \cos k\theta_1 d\theta_1 \wedge\cdots\wedge d\theta_{n-1} \right| \\
& & + (2\pi )^{n+2} C_3 r_1^4 .
\end{eqnarray*}

Combine this and the case $n=1$, we obtain $(ii)$ by induction on $n$.
\qed


\end{document}